\documentclass[a4paper,11pt, reqno]{amsart}
\usepackage{amssymb,amsmath}
\usepackage{cite}
\usepackage{amsmath,amscd,amssymb}
\usepackage{latexsym}
\usepackage[colorlinks,citecolor=blue,pagebackref,hypertexnames=false]{hyperref}

\numberwithin{equation}{section}
\allowdisplaybreaks[2]
\theoremstyle{plain}
\newtheorem{theorem}{Theorem}[section]

\newtheorem{lemma}[theorem]{Lemma}
\newtheorem{corollary}[theorem]{Corollary}
\newtheorem{proposition}[theorem]{Proposition}

\theoremstyle{definition}
\newtheorem{definition}[theorem]{Definition}

\theoremstyle{remark}
\newtheorem{remark}[theorem]{Remark}

\newtheorem{case[theorem]}{Case}

\def\norm#1.#2.{\lVert#1\rVert_{#2}}

\title[Decay estimates and Strichartz inequalities on H-type groups]{Decay estimates and Strichartz inequalities for a class of dispersive equations on H-type groups
}

\author{Manli Song}
\author{Jinggang Tan}
\thanks{Jinggang Tan is the corresponding author.}

\address{\endgraf Research\&Development Institute of Northwestern Polytechnical University in Shenzhen, Sanhang Science\&Technology Building, No. 45th, Gaoxin South 9th Road, Nanshan District, Shenzhen City, 518063}
\address{\endgraf School of Mathematics and Statistics, Northwestern Polytechnical University, Xi'an, Shaanxi 710129, China}
\email{mlsong@nwpu.edu.cn}
\address{\endgraf School of Mathematics and Statistics, Jiangxi Normal University, Nanchang, China}
\email{jtan@jxnu.edu.cn}

\date{\today}
\begin{document}
	
	\maketitle

	\allowdisplaybreaks

	\begin{abstract}
	Let $\mathcal{L}$ be the sub-Laplacian on H-type groups and $\phi: \mathbb{R}^+ \to \mathbb{R}$ be a smooth function. The primary objective of the paper is to study the decay estimate for a class of dispersive semigroup given by $e^{it\phi(\mathcal{L})}$. Inspired by earlier work of Guo-Peng-Wang \cite{GPW2008} in the Euclidean space and Song-Yang \cite{SY2023} on the Heisenberg group, we overcome the difficulty arising from the non-homogeneousness of $\phi$ by frequency localization, which is based on the non-commutative Fourier transform on H-type groups, the properties of the Laguerre functions and Bessel functions, and the stationary phase theorem. Finally, as applications, we derive the new Strichartz inequalities for the solutions of some specific equations, such as the fractional Schr\"{o}dinger equation, the fourth-order Schr\"odinger equation, the beam equation and the Klein-Gordon equation, which corresponds to $\phi(r)=r^\alpha$, $r^2+r,\sqrt{1+r^2},\sqrt{1+r}$, respectively. Moreover, we also prove that the time decay is sharp in these cases.
	\end{abstract}
	\tableofcontents 	
	
	\section{Introduction}
The key ingredient of this paper is to establish a decay estimate for a class of dispersive equations on H-type groups $G$
	\begin{equation}\label{HSEquation}
	\begin{cases}
	i\partial_tu+\phi(\mathcal{L}) u=f,\\
	u|_{t=0}=u_0,
	\end{cases}
	\end{equation}
where $\mathcal{L}$ is the sub-Laplacian on H-type groups $G$, and $\phi: \mathbb{R}^+ \to \mathbb{R}$ is smooth satisfying:
\\

(H1)~There exists $m_1>0$, such that for any $\alpha \geqslant 2$ and $\alpha \in \mathbb{N}$,
\begin{equation*}
|\phi'(r)| \sim r^{m_1-1}  , \quad |\phi^{(\alpha)}(r)| \lesssim r^{m_1-\alpha},\quad r \geqslant 1.
\end{equation*}

(H2)~There exists $m_2>0$, such that for any $\alpha \geqslant 2$ and $\alpha \in \mathbb{N}$,
\begin{equation*}
|\phi'(r)| \sim r^{m_2-1}  , \quad |\phi^{(\alpha)}(r)| \lesssim r^{m_2-\alpha},\quad 0<r<1.
\end{equation*}

(H3)~There exists $\alpha_1>0$, such that
\begin{equation*}
|\phi''(r)| \sim r^{\alpha_1-2}, \quad r \geqslant1.
\end{equation*}

(H4)~There exists $\alpha_2>0$, such that
\begin{equation*}
|\phi''(r)| \sim r^{\alpha_2-2}, \quad  0<r<1.
\end{equation*}
\begin{remark}
 The assumptions (H1)-(H4) root in \cite{GPW2008}. (H1) and (H3) represent the homogeneous order of $\phi$ in high frequency and guarantee $\alpha_1\leq m_1$. Similarly, the homogeneous order of $\phi$ in low frequency is described by (H2) and (H4) which also make sure  $\alpha_2\geq m_2$.
 \end{remark}
 \begin{remark}
 Here we want to emphasize that many dispersive wave equations associated with the sub-Laplacian on H-type groups are reduced to the form \eqref{HSEquation}. For instance, the Schr\"{o}dinger equation corresponds to $\phi(r)=r^2$, the wave equation corresponds to $\phi(r)=r$, the fractional Schr\"{o}dinger equation corresponds to $\phi(r)=r^\alpha$ $(0<\alpha<2$ and $\alpha\neq1)$,  the fourth-order Schr\"{o}dinger equation corresponds to $\phi(r)=r^2+r^4$, the beam equation corresponds to $\phi(r)=\sqrt{1+r^4}$ and the Klein-Gordon equation corresponds to $\phi(r)=\sqrt{1+r^2}$, etc.
 \end{remark}

Strichartz estimates are a type of space-time decay
estimates for solutions to envolution equations whose common property is that their solutions tend to disperse over time. The significance of Strichartz estimates is evident in many ways. Such estimates were originally proved by Strichartz \cite{Str} in the late 1970s as an important application of
restriction estimates for quadratic surfaces, which play a fundamental role in the field of classical harmonic analysis. Strichartz estimates are also
powerful tools to study nonlinear versions of these equations, such as the local and global well-posedness problems for nonlinear partial differential equations (see for instance \cite{Tao2006, BCD} and the
references therein). Furthermore, Strichartz estimates associated to wave equations have been used to prove pointwise Fourier convergence (see \cite{PT}).

Strichartz estimates in the Euclidean setting have been proved for many dispersive equations, such as the wave equation and Schr\"{o}dinger equation (see for instance the pioneering works \cite{GV, KT, Str}). To obtain Strichartz estimates, it involves basically two types of ingredients. The first one consists in estimating the decay in time on the free evolution equation of the form
\begin{align*}\left\|e^{i t \phi(\sqrt{-\Delta})} u_0\right\|_X \lesssim|t|^{-\theta}\left\|u_0\right\|_{X^*},\end{align*}
where $X^*$ is the dual space of $X$.  Its proof requires more elaborate
techniques involving oscillatory integrals and the application of a
stationary phase theorem. In the past decades, dispersive properties of
evolution equations have become a crucial tool in a variety of questions, including local and global existence for nonlinear partial differential equations, the well-posedness of Cauchy problems for nonlinear partial differential equations in Sobolev spaces of low order, scattering theory, and many others.
The second one consists of abstract functional arguments, which are mainly $TT^*$-duality arguments. Therefore, the decay estimate plays a crucial role in the derivation of Strichartz estimates.

In 2008, Guo-Peng-Wang \cite{GPW2008} used a unified way to study the decay for a class of dispersive semigroup $e^{it\phi(\sqrt{-\Delta})}$ on the Euclidean space $\mathbb{R}^d$
\begin{equation*}
	\begin{cases}
	i\partial_tu+\phi(\sqrt{-\Delta}) u=f,\\
	u|_{t=0}=u_0,
	\end{cases}
	\end{equation*}
 where $\Delta=-\sum\limits_{j=1}^d\partial_{x_j}^2$ is the Laplacian on $\mathbb{R}^d$. When $\phi$ is a homogeneous function of order $m$, namely, $\phi(\lambda r)=\lambda^m\phi(r),\,\forall \lambda>0$, the dispersive estimate can be easily obtained by a theorem of Littman and dyadic decomposition. However, it becomes very complicated when $\phi$ is not homogeneous since the scaling constants can not be effectively separated from the time.  To overcome the difficulty, they applied frequency localization by separating $\phi$ between high and low frequency in different scales. Guo-Peng-Wang\cite{GPW2008} assumed that $\phi: \mathbb{R}^+ \to \mathbb{R}$ is smooth satisfying (H1)-(H4) and obtained the following decay estimates in time for the semigroup $e^{it\phi(\sqrt{-\Delta})}$:

	(1)~For $j \geqslant 0$, $\phi$ satisfies (H1), then
	\begin{equation*}
	\left\|e^{it\phi(\sqrt{-\Delta})}\Delta_ju_0\right\|_{L^\infty(\mathbb{R}^d)}\lesssim |t|^{-\theta}2^{j(d-m_1\theta)}\|u_0\|_{L^1(\mathbb{R}^d)},\,0\leq\theta\leq\frac{d-1}{2}.
	\end{equation*}
 In addition, if $\phi$ satisfies (H3), then
	\begin{equation*}
	\left\|e^{it\phi(\sqrt{-\Delta})}\Delta_ju_0\right\|_{L^\infty(\mathbb{R}^d)}\lesssim|t|^{-\frac{d-1+\theta}{2}}2^{j(d-\frac{m_1(d-1+\theta)}{2}-\frac{\theta(\alpha_1-m_1)}{2})}\|u_0\|_{L^1(\mathbb{R}^d)},\, 0 \leq \theta \leq 1.
	\end{equation*}
	
	(2)~For $j < 0$, $\phi$ satisfies (H2), then
		\begin{equation*}
		\left\|e^{it\phi(\sqrt{-\Delta})}\Delta_ju_0\right\|_{L^\infty(\mathbb{R}^d)}\lesssim |t|^{-\theta}2^{j(d-m_2\theta)}\|u_0\|_{L^1(\mathbb{R}^d)},\,0\leq\theta\leq\frac{d-1}{2}.
		\end{equation*}
In addition, if $\phi$ satisfies (H4), then
\begin{equation*}
\left\|e^{it\phi(\sqrt{-\Delta})}\Delta_ju_0\right\|_{L^\infty(\mathbb{R}^d)}\lesssim|t|^{-\frac{d-1+\theta}{2}}2^{j(d-\frac{m_2(d-1+\theta)}{2}-\frac{\theta(\alpha_2-m_2)}{2})}\|u_0\|_{L^1(\mathbb{R}^d)},\, 0 \leq \theta \leq 1.
\end{equation*}

   (3)~If $\phi$ satisfies (H2), then
	\begin{equation*}
    \left\|e^{it\phi(\sqrt{-\Delta})}P_{\leq 0}u_0\right\|_{L^\infty(\mathbb{R}^d)}\lesssim(1+|t|)^{-\theta}\|u_0\|_{L^1(\mathbb{R}^d)},\, \theta=\min \left(\frac{d}{m_2}, \frac{d-1}{2}\right).
	\end{equation*}
  In addition, if $\phi$ satisfies (H4) and $\alpha_2=m_2$, then
	\begin{equation*}
	 \left\|e^{it\phi(\sqrt{-\Delta})}P_{\leq 0}u_0\right\|_{L^\infty(\mathbb{R}^d)}\lesssim(1+|t|)^{-\theta}\|u_0\|_{L^1(\mathbb{R}^d)},\, \theta=\min \left(\frac{d}{m_2}, \frac{d}{2}\right).
	\end{equation*}
Here we choose $\Phi: \mathbb{R}^d\rightarrow [0,1]$ being an even, smooth radial function such that $supp\, \Phi\subseteq \{x\in\mathbb{R}^d: |x|\leq 2\}$ and $\Phi(x)=1,\,\forall |x|\leq 1$. Define $\Phi_j(x)=\Phi(2^{-j}x)-\Phi(2^{-j+1}x)$ and $\{\Phi_j\}_{j\in\mathbb{Z}}$ is the Littlewood-Paley dyadic partition of unity on $\mathbb{R}^d$. Let $\Delta_j=\mathcal{F}^{-1}\Phi_j\mathcal{F}$ be the Littlewood-Paley projector and $P_{\leq 0}=\mathcal{F}^{-1}\Phi\mathcal{F}$, where $\mathcal{F}$ is the Fourier transform on $\mathbb{R}^d$. Moreover, by using standard duality arguments, they applied the above estimates to some concrete  equations and established the corresponding Strichartz estimates. Their results could not only cover known results but also make some improvements and provide simple proof for the wellposedness of some nonlinear dispersive equations.

Many authors are also interested in adapting the well known Strichartz estimates from the Euclidean setting to a more abstract setting, such as H-type groups, metric measure spaces, hyperbolic spaces (noncompact and negatively curved), compact Riemannian manifolds and bounded domains (with
a possible loss of derivatives). In 2000, Bahouri-G\'{e}rard-Xu \cite{BGX2000} initially discussed the Strichartz estimates associated with the sub-Laplacian on the Heisenberg group, by means of Besov spaces defined by a Littlewood-Paley decomposition related to the spectral of the sub-Laplacian. In their work, they showed such estimates existed for the wave equation, while surprisingly failed for the Schr\"{o}dinger equation. To avoid the particular behavior of the Schr\"{o}dinger operator, one way is to replace the Heisenberg group by a bigger space, like H-type groups whose center dimension is bigger than $1$, on which Hierro \cite{H2005} proved that the Schr\"{o}dinger equation is dispersive. For more results on H-type groups and more generality of step $2$ stratified Lie groups, see \cite{BKG, Song2016, SZ}.  Another way is to adapt the Fourier transform restriction analysis initiated by Strichartz \cite{Str} in the  Euclidean space. Using this approach, Bahouri-Barilari-Gallagher \cite{BBG2021} obtain an anisotropic Strichartz estimate for the Schr\"{o}dinger operator on the Heisenberg group. Inspired by Guo-Peng-Wang \cite{GPW2008}, Song-Yang \cite{SY2023} discussed the decay for a class of dispersive semigroup on the Heisenberg group and as applications they obtained Strichartz inequalites for solutions of some specific equations, such as the frational Schr\"{o}dinger equation, the fractional wave equation and the fourth-order Schr\"{o}dinger equation. For more results on the Heisenberg group, see \cite{FMV1, FV, LS2014}. There are also a number of works considered involving Strichartz estimates in the abstract setting of metric measure spaces (see \cite{BDXM2019, DPR2010, FSW, NR2005, R2008, S2013}), hyperbolic spaces (see \cite{AP2009, AP2014, APV2012}), compact Riemannian manifolds (see \cite{B1993, BGT2004}) and bounded domains (see \cite{ILP2014}), etc.

Our main aim is to establish the decay estimate for a class of dispersive equations \eqref{HSEquation} on H-type groups $G$. The proof originates from the combination of Hierro \cite {H2005}, Guo-Peng-Wang \cite{GPW2008} and Song-Yang \cite{SY2023}. By the non-commutative Fourier analysis on H-type groups, the dispersive estimate is reduced to a sum of oscillatory integrals involving a series of Laguerre functions in $p$-dimension where $p$ is the dimension of the center of the H-type groups $G$. The main difficulties are that the Fourier transform on H-type groups is more complicated than that in the Euclidean case, the scaling constants can not be effectively separated from the time $t$ when $\phi$ is not homogeneous, and the high center dimension $p>1$ involves a series of high-dimensional oscillatory integrals. We overcome these difficulties by non-commutative Fourier transform on H-type groups, the frequency localization method, the properties of the Laguerre functions and Bessel functions, and the stationary phase theorem. Our main results are in the following.

\begin{theorem}\label{ResultTime}
	Assume $\phi:\mathbb{R}^+ \to \mathbb{R}$ is smooth. $U_t=e^{it\phi(\mathcal{L})}$. $\varphi_j$ and $\psi$ are the kernels related to the Littlewood-Paley decomposition introduced in Section 2. $N=2d+2p$ is the homogeneous dimension of the H-type group $G$ with underlying manifold $\mathbb{R}^{2d+p}$. Then the following results hold true.
	
	$(a)$~For $j \geq 0$, $\phi$ satisfies (H1), then
	\begin{equation}\label{res3-2}
	\|U_t \varphi_j\|_{L^\infty(G)} \leq C_{d,p}|t|^{-\theta}2^{j\left(N-2m_1\theta\right)},\,0 \leq \theta \leq \frac{p-1}{2}.
	\end{equation}

In addition, if $\phi$ satisfies (H3), then
	\begin{equation}\label{res3-3}
	\|U_t \varphi_j\|_{L^\infty(G)} \leq C_{d,p}|t|^{-\frac{p-1+\theta}{2}}2^{j\left(N-m_1(p-1+\theta)-\theta(\alpha_1-m_1)\right)},\,   0 \leq \theta \leq 1.
	\end{equation}
	
	$(b)$~For $j < 0$, $\phi$ satisfies (H2), then
	\begin{equation*}
	\|U_t \varphi_j\|_{L^\infty(G)} \leq C_{d,p}|t|^{-\theta}2^{j\left(N-2m_2\theta\right)},\,0 \leq \theta \leq \frac{p-1}{2}.
	\end{equation*}

In addition, if $\phi$ satisfies (H4), then
	\begin{equation*}
	\|U_t \varphi_j\|_{L^\infty(G)} \leq C_{d,p}|t|^{-\frac{p-1+\theta}{2}}2^{j\left(N-m_2(p-1+\theta)-\theta(\alpha_2-m_2)\right)},\,   0 \leq \theta \leq 1.
	\end{equation*}

    $(c)$~If $\phi$ satisfies (H2), then
	\begin{equation}\label{res-sum1}
 \|U_t \psi\|_{L^\infty(G)} \leq C_{d,p}(1+|t|)^{-\theta},\,\theta=\min\left(\frac{N}{2m_2},\frac{p-1}{2}\right).
	\end{equation}

In addition, if $\phi$ satisfies (H4) and $\alpha_2=m_2$, then
	\begin{equation}\label{res-sum2}
	\|U_t \psi\|_{L^\infty(G)} \leq C_{d,p}(1+|t|)^{-\theta},\,\theta=\min\left(\frac{N}{2m_2},\frac{p}{2}\right);
	\end{equation}
 if $\phi$ satisfies (H4), $\alpha_2>m_2$ and $p<\frac{N-(\alpha_2-m_2)}{m_2}$, then
 \begin{equation}\label{res-sum2-n}
	\|U_t \psi\|_{L^\infty(G)} \leq C_{d,p}(1+|t|)^{-\frac{p}{2}}.
	\end{equation}
\end{theorem}
Finally, as applications, in Section $5$ we shall derive the new Strichartz inequalities for the solutions of some specific equations, such as the fractional Schr\"{o}dinger equation, the fourth-order Schr\"odinger equation, the beam equation and the Klein-Gordon equation, which corresponds to $\phi(r)=r^\alpha$, $r^2+r,\sqrt{1+r^2},\sqrt{1+r}$, respectively.  

This paper is organized as follows: In Section 2, we introduce some basic notations and definitions on H-type groups. We also
recall the properties for spherical Fourier transform and the homogeneous Besov spaces on H-type groups.  Section
3 is dedicated to describe the useful technical lemmas for proving the decay estimates and Strichartz estimates. In Section 4, we prove our main results Theorem \ref{ResultTime}. In Section 5, we exploit Theorem \ref{ResultTime} to obtain the Strichartz estimates for some concrete equations, such as the fractional Schr\"{o}dinger equation, the fourth-order Schr\"{o}dinger equation, the beam equation and the Klein-Gordon equation. It is worth noticing that the list of applications is not exhaustive since we just intend to show the generality of our approach.

\begin{remark}
	 Throughout this paper, $A\lesssim B$ means that $A\leqslant CB$, and $A\sim B$ stands for $C_1B\leq A\leq C_2B$, where $C$, $C_1$, $C_2$ denote positive universal constants.
\end{remark}

\section{Preliminaries}
\subsection{H-type groups}
Let $\mathfrak{g}$ be a two step nilpotent Lie algebra endowed with an inner product $\langle \cdot,\cdot \rangle$. Its center is denoted by $\mathfrak{z}$. $\mathfrak{g}$ is said to be of H-type if $[\mathfrak{z}^{\bot},\mathfrak{z}^{\bot}]=\mathfrak{z}$ and for every $s \in \mathfrak{z}$, the map $J_s: \mathfrak{z}^{\bot} \rightarrow \mathfrak{z}^{\bot}$ defined by
\begin{equation*}
\langle J_s u, w \rangle:=\langle s, [u,w] \rangle, \quad\forall u, w \in \mathfrak{z}^{\bot},
\end{equation*}
is an orthogonal map whenever $|s|=1$.\\
\indent An H-type group is a connected and simply connected Lie group $G$ whose Lie algebra is of H-type.\\
\indent Given $0 \neq a \in \mathfrak{z}^*$, the dual of $\mathfrak{z}$, we can define a skew-symmetric mapping $B(a)$ on $\mathfrak{z}^{\bot}$ by
\begin{equation*}
\langle B(a)u,w \rangle =a([u,w]), \forall u,w \in \mathfrak{z}^{\bot}.
\end{equation*}
We denote by $z_a$ the element of $\mathfrak{z}$ determined by
\begin{equation*}
\langle B(a)u,w \rangle =a([u,w])=\langle J_{z_a} u,w \rangle.
\end{equation*}
Since $B(a)$ is skew symmetric and non-degenerate, the dimension of $\mathfrak{z}^{\bot}$ is even, i.e. dim$\mathfrak{z}^{\bot}=2d$.
We can choose an orthonormal basis
\begin{equation*}
\{E_1(a),E_2(a),\cdots,E_d(a),\overline{E}_1(a),\overline{E}_2(a),\cdots,\overline{E}_d(a)\}
\end{equation*}
of $\mathfrak{z}^{\bot}$ such that
\begin{equation*}
B(a)E_i(a)=|z_a|J_{\frac{z_a}{|z_a|}}E_i(a)=|a|\overline{E}_i(a)
\end{equation*}
and
\begin{equation*}
B(a)\overline{E}_i(a)=-|a|E_i(a).
\end{equation*}
We set $p=dim \mathfrak{z}$. We can choose an orthonormal basis $\{\epsilon_1,\epsilon_2,\cdots,\epsilon_p \}$ of $\mathfrak{z}$ such that $a(\epsilon_1)=|a|,a(\epsilon_j)=0,j=2,3,\cdots,p$. Then we can denote the element of $\mathfrak{g}$ by
\begin{equation*}
(z,t)=(x,y,t)=\underset{i=1}{\overset{d}{\sum}}(x_i E_i+y_i \overline{E}_i )+\underset{j=1}{\overset{p}{\sum}}s_j \epsilon_j.
\end{equation*}
We identify $G$ with its Lie algebra $\mathfrak{g}$ by exponential map. The group law on H-type group $G$ has the form
\begin{equation}
(z,s)(z',s')=(z+z',s+s'+\frac{1}{2}[z,z']),  \label{equ:Law}
\end{equation}
where $[z,z']_j=\langle z,U^jz' \rangle$ for a suitable skew symmetric matrix $U^j,j=1,2,\cdots,p$.
\begin{theorem} \indent G is an H-type group with underlying manifold $\mathbb{R}^{2d+p}$, with the group law  $\eqref{equ:Law}$ and the matrix $U^j,j=1,2,\cdots$,p, satisfies the following conditions:\\
$(i)$ $U^j$ is a $2d \times 2d$ skew symmetric and orthogonal matrix, $j=1,2,\cdots$,p.\\
$(ii)$ $U^i U^j+U^j U^i=0,i,j=1,2,\cdots,p$ with $i \neq j$.
\end{theorem}
{\bf Proof.} See \cite{BU}.
\begin{remark}\label{p-d} It is well known that H-type algebras are closely related to Clifford modules (see \cite{R}). H-type algebras can be classified by the standard theory of Clifford algebras. Specially, on H-type group $G$, there is a relation between the dimension of the center and its orthogonal complement space. That is $p+1\leq 2d$ (see \cite{KR}).
\end{remark}
\begin{remark}
We identify $G$ with $\mathbb{R}^{2d}\times\mathbb{R}^p$ and denote by $n=2d+p$ its topological dimension. Following Folland and Stein (see \cite{FS}), we will exploit the canonical homogeneous structure, given by the family of dilations $\{\delta_r\}_{r>0}$,
\begin{equation*}
\delta_r(z,s)=(rz,r^2s).
\end{equation*}
We then define the homogeneous dimension of $G$ by $N=2d+2p$.
\end{remark}
The left invariant vector fields which agree respectively with $\frac{\partial}{\partial x_j},\frac{\partial}{\partial y_j}$ at the origin are given by
\begin{equation*}
\begin{aligned}
X_j&=\frac{\partial}{\partial x_j}+\frac{1}{2}\underset{k=1}{\overset{p}{\sum}} \left( \underset{l=1}{\overset{2d}{\sum}}z_l U_{l,j}^k \right) \frac{\partial}{\partial s_k},\\
Y_j&=\frac{\partial}{\partial y_j}+\frac{1}{2}\underset{k=1}{\overset{p}{\sum}} \left( \underset{l=1}{\overset{2d}{\sum}}z_l U_{l,j+d}^k \right) \frac{\partial}{\partial s_k},\\
\end{aligned}
\end{equation*}
where $z_l=x_l,z_{l+d}=y_l,l=1,2,\cdots,d$. In terms of these vector fields we introduce the sublaplacian $\Delta$ by
\begin{equation*}
\mathcal{L}=-\underset{j=1}{\overset{d}{\sum}}(X_j^2 +Y_j^2).
\end{equation*}
\noindent\\[2mm]
\subsection{Spherical Fourier transform}
Kor\'{a}nyi \cite{Kor}, Damek and Ricci \cite{DR} have computed the spherical functions associated to the Gelfand pair $(G, O(d))$ (we identify $O(d)$ with $O(d)\otimes Id_p$). They involve, as on the Heisenberg group, the Laguerre functions
\begin{equation*}
\mathfrak{L}_m^{(\gamma)}(\tau)=L_m^{(\gamma)}(\tau)e^{-\tau/2}, \quad \tau \in \mathbb{R}, m,\gamma \in \mathbb{N},
\end{equation*}
where $L_m^{(\gamma)}$ is the Laguerre polynomial of type $\gamma$ and degree $m$.

We say a function $f$ on $G$ is radial if the value of $f(z,s)$ depends only on $|z|$ and $s$.  We denote by $\mathcal{S}_{rad}(G)$ and $L^q_{rad}(G)$,$1\leq q\leq \infty$, the spaces of radial functions in $\mathcal{S}(G)$ and $L^p(G)$, respectively. In particular, the set of $L^1_{rad}(G)$ endowed with the convolution product
\begin{equation*}
f_1*f_2(g)=\int_Gf_1(gg'^{-1})f_2(g')\,dg', \quad g\in G
\end{equation*}
is a commutative algebra.

Let $f\in L^1_{rad}(G)$. We define the spherical Fourier transform, $ m\in\mathbb{N}, \lambda\in\mathbb{R}^p$,
\begin{equation*}
\hat{f}(\lambda,m)=\left( \begin{array}{c} m+d-1\\m \end{array} \right)^{-1}
\int_{\mathbb{R}^{2d+p}}e^{i\lambda s} f(z,s)\mathfrak{L}_m^{(d-1)}(\frac{|\lambda|}{2}|z|^2)\,dzds.
\end{equation*}
By a direct computation, we have $\widehat{f_1*f_2}=\hat{f_1}\cdot\hat{f_2}$. Thanks to a partial integration on the sphere $\mathbb{S}^{p-1}$, we deduce from the Plancherel theorem on the Heisenberg group its analogue for the H-type groups.

\begin{proposition} \indent For all $f \in \mathcal{S}_{rad}(G)$ such that
\begin{equation*}
\underset{m\in\mathbb{N}}{\sum}\left( \begin{array}{c} m+d-1\\m \end{array} \right) \int_{\mathbb{R}^p} |\hat{f}(\lambda,m)||\lambda|^d d\lambda <\infty,
\end{equation*}
we have
\begin{equation}\label{plancherel}
f(z,s)=(\frac{1}{2\pi})^{d+p}\underset{m\in\mathbb{N}}{\sum}\int_{\mathbb{R}^p} e^{-i\lambda\cdot s} \hat{f}(\lambda,m) \mathfrak{L}_m^{(d-1)}\left(\frac{|\lambda|}{2}|z|^2\right)|\lambda|^d \,d\lambda,
\end{equation}
the sum being convergent in $L^{\infty}$ norm.
\end{proposition}
Moreover, if $f \in \mathcal{S}_{rad}(G)$, the functions $\mathcal{L} f$ is also in $\mathcal{S}_{rad}(G)$ and its spherical Fourier transform is given by
\begin{equation*}
\widehat{\mathcal{L} f}(\lambda,m)=(2m+d)|\lambda|\hat{f}(\lambda,m).
\end{equation*}
The sublaplacian $\mathcal{L}$ is a positive self-adjoint operator densely defined on $L^2(G)$. So by the spectral theorem (see \cite{Hul}), for any bounded Borel function $h$ on $\mathbb{R}$, we have
\begin{equation*}
\widehat{h(\mathcal{L})f}(\lambda,m)=h((2m+d)|\lambda|)\hat{f}(\lambda,m).
\end{equation*}

\subsection{Besov spaces}\label{HBS}
In this section, we shall recall the Besov spaces related to the sub-Laplacian $\mathcal{L}$ on the H-type group $G$ given in \cite{H2005}.

Fix a non-increasing smooth function $\Psi$ on $[0,+\infty)$ such that $\Psi=1$ in $[0,2]$ and $\Psi=0$ in $[4,+\infty)$. Let $R(\tau)=\Psi(\tau)-\Psi(4\tau)$ and $0\leq R\leq 1$ such that $supp\, R\subseteq C_0=[\frac{1}{2},4]$. Moreover, we observe that
	\begin{align*}
\Psi(\tau)+\sum_{j=1}^\infty R(2^{-2j}\tau)&=1, \quad \forall \tau\geq0,\\
\underset{j\in \mathbb{Z}}{\sum}R(2^{-2j}\tau)&=1, \quad \forall \tau>0.
	\end{align*}

Since both $\Psi$ and $R$ are in $\mathcal{S}(\mathbb{R}^+)$, by \cite{FMV2006},  we have the kernels $\psi$ and $\varphi$ of the operators $P_{j\leq 0}:=\Psi(\mathcal{L})$ and $R(\mathcal{L})$ belong to $\mathcal{S}_{rad}(G)$. In particular,
	\begin{equation*}
	\widehat{\varphi}(m, \lambda)=R_m(|\lambda|):=R((2m+n)|\lambda|).
	\end{equation*}
By the inversion Fourier transform \eqref{plancherel}, we also have 	
\begin{equation*}
	\varphi(z,s)=(\frac{1}{2\pi})^{d+p}\underset{m\in\mathbb{N}}{\sum}\int_{\mathbb{R}^p} e^{-i\lambda\cdot s} R_m(|\lambda|) \mathfrak{L}_m^{(d-1)}\left(\frac{|\lambda|}{2}|z|^2\right)|\lambda|^d \,d\lambda.
	\end{equation*}
	
For $j\in \mathbb{Z}$, we denote by $\varphi_j$ the kernel of the operator $R(2^{-2j}\mathcal{L})$. It can be easily calculated that
$\varphi_j(z,s)=2^{Nj}\varphi(\delta_{2^j}(z,s))$, $\forall (z,t)\in G$. The Littlewood-Paley projector on H-type groups is define by
	\begin{equation*}
	\Delta_jf=R(2^{-2j}\mathcal{L})f=f*\varphi_j,  \quad \forall f \in \mathcal{S}'(G).
	\end{equation*}
Obviously $\widehat{\varphi_j}(m,\lambda)=R_m(2^{-2j}|\lambda|):=R^j_m(|\lambda|)$. Let $\widetilde{\varphi_j}=\varphi_{j-1}+\varphi_j+\varphi_{j+1}$.  We have  $f*\varphi_j=f*\varphi_j*\widetilde{\varphi_j}$ holds true for all $f \in \mathcal{S}'(G)$.

By the spectral theorem, for any $f \in L^2(G)$, the following homogeneous Littlewood-Paley decomposition holds:
	\begin{equation*}
	f=\sum_{j\in\mathbb{Z}}\Delta_j f \quad \text{in $L^2(G)$}.
	\end{equation*}
So
	\begin{equation}\label{infty}
	\|f\|_{L^\infty(G)}\leq\sum_{j\in\mathbb{Z}}\|\Delta_j f\|_{L^\infty(G)}, \, f\in L^2(G),
	\end{equation}
where both sides of \eqref{infty} are allowed to be infinite.

In addition, it has been proved in \cite{FMV2006} that
for any $\sigma\in\mathbb{R}$, $j\in\mathbb{Z}$, $1\leq q\leq \infty$ and $f\in \mathcal{S}'(G)$, then
	\begin{equation}\label{LPLP}
	\left\|\mathcal{L}^{\frac{\sigma}{2}}\Delta_jf\right\|_{L^q(G)} \leq C2^{j\sigma}\|\Delta_jf\|_{L^q(G)}.
	\end{equation}

\begin{definition}
Let $1\leq q,r \leq \infty, \rho \in\mathbb{R}$. The Besov space $B^\rho_{q,r}(G)$ is defined as the set of distributions $f \in \mathcal{S}'(G)$ such that
	\begin{equation*}
	\|f\|_{B^\rho_{q,r}(G)}=\|P_{j\leq 0}f\|_{L^q(G)}+\left(\sum_{j=1}^\infty2^{j\rho r}\|\Delta_j\|_{L^q(G)}^r\right)^{\frac{1}{r}}<\infty,
	\end{equation*}
and $f=\underset{j\in \mathbb{Z}}{\sum}\Delta_jf$ in $\mathcal{S}'(G)$.
\end{definition}

 \begin{definition}
Let $\rho \in\mathbb{R}$. The Sobolev space $H^\rho(G)$ is
	\begin{equation*}
	H^\rho(G)=B^\rho_{2,2}(G),
	\end{equation*}
which is equivalent to
	\begin{equation*}
	u\in H^\rho(G)\Leftrightarrow u,\mathcal{L}^{\rho/2}u\in L^2(G).
	\end{equation*}
and the associated norms are of course equivalent.
\end{definition}
Analogous to the Besov space in the Euclidean setting, we list some properties of the space $B^\rho_{q,r}(G)$ in the following proposition.
\begin{proposition}[see \cite{FMV2006}]\label{properties} Let $p,r\in [1,\infty]$ and $\rho \in\mathbb{R}$.
	
(i) The space $B^\rho_{q,r}(G)$ is a Banach space with the norm $||\cdot||_{B^\rho_{q,r}(G)}$;

(ii) the definition of $B^\rho_{q,r}(G)$ does not depend on the choice of the function $\Phi$ in the Littlewood-Paley decomposition;

(iii) the dual space of $B^\rho_{q,r}(G)$ is $B^{-\rho}_{q',r'}(G)$;

(iv) for any $\sigma>0$, the continuous inclusion holds $B^{\rho+\sigma}_{q,r}(G)\subseteq B^{\rho}_{q,r}(G)$;

(v) for any $u\in \mathcal{S}'(G)$ and $\sigma\in\mathbb{R}$, then $u\in B^\rho_{q,r}(G)$ if and only if $\mathcal{L}^{\sigma/2}u\in B^{\rho-\sigma}_{q,r}(G)$;

(vi) for any $q_1,q_2\in[1,\infty]$, the continuous inclusion holds
\begin{equation*}
B^{\rho_1}_{q_1,r}(G)\subseteq B^{\rho_2}_{q_2,r}(G), \quad \frac{1}{q_1}-\frac{\rho_1}{N}=\frac{1}{q_2}-\frac{\rho_2}{N}, \quad \rho_1\geq\rho_2;
\end{equation*}

(vii) for all $q\in[2, \infty]$ we have the continuous inclusion $B^0_{q,2}(G)\subseteq L^q(G)$;\\

(viii) $B^0_{2,2}(G)=L^2(G)$.
\end{proposition}

\section{Technical Lemmas}
By the inversion Fourier formula \eqref{plancherel}, we may write $U_t\varphi_j$ explicitly into a sum of a list of oscillatory integrals in high dimension. In order to estimate the oscillatory integrals, we recall the stationary phase lemma.
\begin{lemma}[see \cite{S1993}]\label{phase} Let $g\in C^\infty([a,b])$ be real-valued such that
	\begin{equation*}
	|g''(x)|\geq \delta
	\end{equation*}
for any $x\in[a,b]$ with $\delta >0$. Then for any function $\psi \in C^\infty([a,b])$, there exists a constant $C$ which does not depend on $\delta, a, b, g$ or $\psi$, such that
	\begin{equation*}
	\left|\int_a^b e^{ig(x)}\psi(x)\,dx\right|\leq C\delta^{-1/2}\left(\|\psi\|_\infty+\|\psi'\|_1\right).
	\end{equation*}
\end{lemma}

In order to prove the sharpness of the time decay dispersion of the solutions of some concrete equations in Section 5, we describe the asymptotic expansion of oscillating integrals.
\begin{lemma}[see \cite{S1993}]\label{asymptotic} Suppose $g(x_0)=0$, and $g$ has a nondegenerate critical point at $x_0$. If $\psi$ is supported in a sufficiently neighborhood of $x_0$, then
	\begin{equation*}
	\int_{\mathbb{R}^p}e^{itg(x)}\psi(x)\,dx\sim |t|^{-p/2}\sum_{j=0}^{\infty}a_j|t|^{-j}, \quad \text{ as t}\rightarrow \infty.\\
	\end{equation*}
\end{lemma}

Besides, it will involve Laguerre functions when we handle the oscillatory integrals. We need the following estimates.
\begin{lemma}[see \cite{H2005}]\label{Laguerre} Consider the following:
	\begin{equation*}
	\left|(\tau \frac{d}{d\tau})^\beta \mathfrak{L}_m^{(d-1)}(\tau)\right|\leq C_{\beta,d}(2m+d)^{d-1/4}
	\end{equation*}
for all $0\leq \beta\leq d$.
\end{lemma}
\begin{remark}\label{better-Laguerre}
In fact, from the proof of Lemma 3.2 in \cite{H2005},  we have a better estimate
	\begin{equation*}
	\left|(\tau \frac{d}{d\tau})^\beta \mathfrak{L}_m^{(d-1)}(\tau)\right|\leq C_{\beta,d}(2m+d)^{d-1},
	\end{equation*}
for $0\leqslant \beta \leqslant d-1$.
\end{remark}

By changing to polar coordinates, We reduce oscillatory integrals in high dimensional to those in one dimension relating Bessel functions. Let $J_\nu$ be the Bessel function of order $\nu>-\frac{1}{2}$,
\begin{equation*}
J_\nu(r)=\frac{(r/2)^\nu}{\Gamma(\nu+1/2)\pi^{1/2}}\int_{-1}^1e^{ir\tau}(1-\tau^2)^{\nu-1/2}d\tau.
\end{equation*}
We list some properties of $J_\nu$ in the following lemma.
\begin{lemma}\label{Bessel}
For $r>0$ and $\nu>-\frac{1}{2}$, we have
\begin{align}
&(1)J_\nu(r)\leq C_\nu r^\nu,\label{bessel1}\\
&(2)\frac{d}{dr}\left(r^{-\nu}J_\nu(r)\right)=-r^{-\nu}J_{\nu+1}(r),\label{bessel2}\\
&(3)J_\nu(r)\leq C_\nu r^{-1/2}.\label{bessel3}
\end{align}
\end{lemma}
\begin{remark}
From Lemma \ref{Laguerre}, \eqref{bessel1} and \eqref{bessel2} of Lemma \ref{Bessel}, we can easily obtain that for any $0\leq \tau\leq 2$ and $0\leq \beta\leq d$,
\begin{equation}\label{small-s}
\left|\frac{d^\beta}{dr^\beta}\left((r\tau)^{-\frac{p-2}{2}}J_\frac{p-2}{2}(r\tau) R(r)\mathfrak{L}_m^{(d-1)}\left(\frac{r}{2}|z|^2\right)r^{d+p-1}\right)\right|\leq C_{\beta,d,p}(2m+d)^{d-1/4},
\end{equation}
where $R$ is the function in Section \ref{HBS}.\\
It is known that $J_\frac{p-2}{2}(r)$ is closely related to the Fourier transform of the spherical surface measure on $\mathbb{S}^{p-1}$, which is
\begin{equation*}
    \widehat{d\sigma}(\xi)=\int_{\mathbb{S}^{p-1}} e^{-i\varepsilon\cdot \xi} d\sigma(\varepsilon)=2\pi \left(\frac{|\xi|}{2\pi}\right)^{-\frac{p-2}{2}}J_\frac{p-2}{2}(|\xi|).
\end{equation*}
By Chapter 1, Eq. (1.5) of \cite{John1981}, we have
\begin{equation}\label{Bessel-Fourier}
    r^{-\frac{p-2}{2}}J_\frac{p-2}{2}(r)=C_p\left(e^{ir}h_+(r)+e^{-ir}h_-(r)\right),
\end{equation}
where $h_\pm$ satisfies
\begin{equation}\label{h}
\left|\frac{d^\beta}{dr^\beta}h_\pm(r)\right|\leq C_{\beta,p}(1+r)^{-\frac{p-1}{2}-\beta}, \forall \beta\in\mathbb{N}.
\end{equation}
Therefore, from Lemma \ref{Laguerre} and \eqref{h}, for any $\tau>2$ and $0\leq\beta\leq d$, we have
\begin{equation}\label{big-s}
\left|\frac{d^\beta}{dr^\beta}\left(h_\pm(r\tau) R(r)\mathfrak{L}_m^{(d-1)}\left(\frac{r}{2}|z|^2\right)r^{d+p-1}\right)\right|\leq C_{\beta,d,p} \tau^{-\frac{p-1}{2}}(2m+d)^{d-1/4}.
\end{equation}
\end{remark}

We also exploit the following estimates which can be easily proved by comparing the sums with the corresponding integrals.
\begin{lemma}\label{Sum}
Fix $\beta>0$. There exists $C_\beta>0$ such that for any $A>0$, we have
\begin{align*}
\sum_{j\in\mathbb{Z}, 2^j\leq A}2^{j\beta}&\leq C_\beta A^\beta,\\
\sum_{j\in\mathbb{Z}, 2^j>A}2^{-j\beta}&\leq C_\beta A^{-\beta}.
\end{align*}
\end{lemma}
Finally, we apply the following duality arguments to prove the Strichartz estimates for some concrete equations on H-type groups.

\begin{lemma} [see \cite{GV}]\label{equ} Let $H$ be a Hilbert space, $X$ a Banach space,$X^*$ the dual of $X$, and $D$ a vector space densely contained in $X$. Let $A\in \mathcal{L}_a(D,H)$ and let $A^*\in \mathcal{L}_a(H,D_a^*)$ be its adjoint operator, defined by
	\begin{equation*}
	\langle A^*v,f \rangle_D=\langle v,Af \rangle_H, \quad \forall f\in D, \quad \forall v \in H,
	\end{equation*}
where $\mathcal{L}_a(Y,Z)$ is the space of linear maps from a vector space $Y$ to a vector space $Z$, $D^*_a$ is the algebraic dual of $D$, $\langle \varphi,f\rangle_D$ is the pairing between $D^*_a$ and $D$, and $\langle\cdot,\cdot\rangle_H$ is the scalar product in $H$. Then the following three condition are equivalent:

(1) There exists $a$, $0\leqslant a \leqslant \infty$ such that for all $f \in D$,
	\begin{equation*}
	\|Af\| \leqslant a\|f\|_X.
	\end{equation*}
	
(2) $\mathfrak{R}(A^*)\subset X^*$, and there exists $a$, $0\leqslant a \leqslant \infty$, such that for all $v \in H$,
	\begin{equation*}
	\|A^*v\|_{X^*} \leqslant a\|v\|.
	\end{equation*}
	
(3) $\mathfrak{R}(A^*A)\subset X^*$,and there exists $a$, $0\leqslant a \leqslant \infty$, such that for all $f \in D$,
	\begin{equation*}
	\|A^*Af\|_{X^*} \leqslant a^2\|f\|_X,
	\end{equation*}
where $||\cdot||$ denote the norm in $H$. The constant $a$ is the same in all three parts. If one of those conditions is satisfied, the operators $A$ and $A^*A$ extend by continuity to bounded operators from $X$ to $H$ and from $X$ to $X^*$,respectively.
\end{lemma}
\begin{lemma}[see \cite{GV}]\label{DuilatyXX} Let $H$,$D$ and two triplets$(X_i,A_i,a_i)$,$i=1,2$, satisfy any of the conditions of Lemma \ref{equ}. Then for all choices of $i,j=1,2$, $\mathfrak{R}(A_i^*A_j) \subset X_i^*$,and for all $f \in D$,
	\begin{equation*}
	\|A_i^* A_jf\|_{X_i^*} \leqslant a_i a_j\|f\|_{X_j}.
	\end{equation*}
\end{lemma}
\begin{lemma}\label{L1LW} (see \cite{GV}) Let $H$ be a Hilbert space, let $I$ be an interval of $\mathbb{R}$, let $X \subset S'(I \times \mathbb{R}^n)$ be a Banach space, let $X$ be stable under time restriction, let $X$ and $A$ satisfy(any of) the conditions of Lemma \ref{equ}. Then the operator $A^*A$ is a bounded operator from $L_t^1(I,H)$ to $X^*$ and from $X$ to $L_t^\infty(I,H)$.
\end{lemma}

\section{Decay Estimates}
In this section, we will prove Theorem\ref{ResultTime}.

\noindent{\bf Proof of Theorem \ref{ResultTime}: }~ First, we prove (a). Since the case $p=1$ has been obtained in \cite{SY2023}, we assume $p\geq 2$. For any $j\geq0$, according to $\varphi_j$ introduced in section 2, we have
	\begin{equation*}
	\begin{aligned}
	U_t\varphi_j(z,s)&=\frac{1}{(2\pi)^{d+p}}\sum_{m=0}^{\infty} \int_{\mathbb{R}^p} e^{-i\lambda\cdot s} e^{it\phi((2m+d)|\lambda|)} \\ & \qquad \qquad \qquad \times R^j_m(|\lambda|) \mathfrak{L}_m^{(d-1)}\left(\frac{|\lambda|}{2}|z|^2\right)|\lambda|^d \,d\lambda.
	\end{aligned}
\end{equation*}
By the change of variable $\lambda'=2^{-2j}(2m+d)\lambda$, we get
	\begin{equation}\label{Imj}
	\begin{aligned}
	U_t\varphi_j(z,s)&=\frac{1}{(2\pi)^{d+p}}\sum_{m=0} ^\infty\frac{2^{Nj}}{(2m+d)^{d+p}}\int_{\mathbb{R}^p} e^{-i\frac{2^{2j}}{2m+d}\lambda' \cdot s} e^{it\phi(2^{2j}|\lambda'|)} \\
	&\qquad \qquad \times R(|\lambda'|) \mathfrak{L}_m^{(d-1)}\left(\frac{2^{2j-1}}{2m+d}|\lambda'||z|^2\right)|\lambda'|^d \,d\lambda' \\
	&=\frac{1}{(2\pi)^{d+p}}\sum_{m=0} ^\infty I_{m,j}(t,z,s),
	\end{aligned}
	\end{equation}
where
	\begin{equation*}
	\begin{aligned}
	I_{m,j}(t,z,s)&=\frac{2^{Nj}}{(2m+d)^{d+p}}\int_{\mathbb{R}^p} e^{-i\frac{2^{2j}}{2m+d}\lambda' \cdot s} e^{it\phi(2^{2j}|\lambda'|)} \\
	          &\qquad \qquad \times R(|\lambda'|) \mathfrak{L}_m^{(d-
    1)}\left(\frac{2^{2j-1}}{2m+d}|\lambda'||z|^2\right)|\lambda'|^d
    \,d\lambda'\\
              &=\frac{2^{Nj}}{(2m+d)^{d+p}}\int_{\mathbb{S}^{p-1}}\int_0^\infty e^{-i\frac{2^{2j}}{2m+d} r \varepsilon\cdot s}  e^{it\phi(2^{2j}r)}\\
              &\qquad \qquad \times R(r) \mathfrak{L}_m^{(d-
    1)}\left(\frac{2^{2j-1}}{2m+d}r|z|^2\right)r^{d+p-1}
    \,dr d\sigma(\varepsilon).
	\end{aligned}
	\end{equation*}
 We integrate first over $\mathbb{S}^{p-1}$:
 \begin{equation}\label{IImj}
	\begin{aligned}
	I_{m,j}(t,z,s)&=\frac{2^{Nj}}{(2m+d)^{d+p}}\int_0^\infty \widehat{d\sigma}\left(\frac{2^{2j}}{2m+d}r s\right)  e^{it\phi(2^{2j}r)}\\
              &\qquad \qquad \times R(r) \mathfrak{L}_m^{(d-
    1)}\left(\frac{2^{2j-1}}{2m+d}r|z|^2\right)r^{d+p-1}
    \,dr\\
    &=\uppercase\expandafter{\romannumeral2}_{m,j}\left(t,\sqrt{\frac{2^{2j}}{2m+d}}z,\frac{2^{2j}}{2m+d}s\right),
	\end{aligned}
	\end{equation}
 where
 \begin{equation*}
	\begin{aligned}
	\uppercase\expandafter{\romannumeral2}_{m,j}(t,z,s)&=\frac{2^{Nj}}{(2m+d)^{d+p}}\int_0^\infty \widehat{d\sigma}\left(rs\right)  e^{it\phi(2^{2j}r)}\\
              &\qquad \qquad \times R(r) \mathfrak{L}_m^{(d-
    1)}\left(\frac{r}{2}|z|^2\right)r^{d+p-1}
    \,dr\\
    &=\frac{(2\pi)^\frac{p}{2}2^{Nj}}{(2m+d)^{d+p}}\int_0^\infty  e^{it\phi(2^{2j}r)}(r |s|)^{-\frac{p-2}{2}}J_\frac{p-2}{2}(r |s|)\\
              &\qquad \qquad \times R(r) \mathfrak{L}_m^{(d-
    1)}\left(\frac{r}{2}|z|^2\right)r^{d+p-1}
    \,dr.
	\end{aligned}
	\end{equation*}
 It suffices to obtain the $L_{z,s}^\infty$-estimate for $\uppercase\expandafter{\romannumeral2}_{m,j}(t,z,s)$.\\
 From Lemma \ref{Laguerre} and \eqref{bessel1} in Lemma \ref{Bessel}, we obtain the trivial estimate
\begin{equation}\label{trivial}
\left|\uppercase\expandafter{\romannumeral2}_{m,j}(t,z,s)\right|\leq C_{d,p} 2^{jN}(2m+d)^{-p-1/4}.
\end{equation}
We will continue the discussion in two cases. On the one hand, we use the vanishing property at the origin \eqref{bessel1} and the recurring property for the Bessel functions \eqref{bessel2}. On the other hand, we apply the decay property of the Bessel function \eqref{Bessel-Fourier}.\\
\textbf{Case 1.} $|s|\leq 2$. In this case, we shall use the vanishing property at the origin. Denote $D_r=\frac{1}{i2^{2j}t\phi'(2^{2j}r)}\frac{d}{dr}$. We see that
\begin{equation*}
    D_r\left(e^{it\phi(2^{2j}r)}\right)=e^{it\phi(2^{2j}r)}, \quad D_r^*f=\frac{i}{2^{2j}t}\frac{d}{dr}\left(\frac{1}{\phi'(2^{2j}r)}f\right)
\end{equation*}
For $\beta\in\mathbb{N}$ and $r\in [1/2,2]$, it follows from (H1) that
\begin{equation}\label{phase-prime}
 \frac{d^\beta}{dr^\beta}\left(\frac{1}{\phi'(2^{2j}r)}\right)\leq C_\beta 2^{-2j(m_1-1)}.
\end{equation}
Using integration by part, for any $0\leq q\leq d$, we have
 \begin{equation*}
	\begin{aligned}
	\uppercase\expandafter{\romannumeral2}_{m,j}(t,z,s)&=\frac{(2\pi)^\frac{p}{2}2^{Nj}}{(2m+d)^{d+p}}\int_0^\infty  e^{it\phi(2^{2j}r)}(r |s|)^{-\frac{p-2}{2}}J_\frac{p-2}{2}(r |s|)\\
              &\qquad \qquad \times R(r) \mathfrak{L}_m^{(d-
    1)}\left(\frac{r}{2}|z|^2\right)r^{d+p-1}
    \,dr\\
    &=\frac{(2\pi)^\frac{p}{2}2^{Nj}}{(2m+d)^{d+p}}\int_0^\infty  D_r\left(e^{it\phi(2^{2j}r)}\right)(r |s|)^{-\frac{p-2}{2}}J_\frac{p-2}{2}(r |s|)\\
              &\qquad \qquad \times R^*(r) \mathfrak{L}_m^{(d-
    1)}\left(\frac{r}{2}|z|^2\right)r^{d+p-1}
    \,dr\\
     &=\frac{(2\pi)^\frac{p}{2}2^{Nj}}{(2m+d)^{d+p}}\cdot \frac{i}{2^{2j}t}\int_0^\infty  e^{it\phi(2^{2j}r)}\frac{d}{dr}\left(\frac{1}{\phi'(2^{2j}r)}(r |s|)^{-\frac{p-2}{2}}J_\frac{p-2}{2}(r |s|)\right.\\
              &\qquad \qquad \left.\times R(r) \mathfrak{L}_m^{(d-
    1)}\left(\frac{r}{2}|z|^2\right)r^{d+p-1}\right)
    \,dr\\
    &=\frac{(2\pi)^\frac{p}{2}2^{Nj}}{(2m+d)^{d+p}}\cdot \left(\frac{i}{2^{2j}t}\right)^q\sum_{k=0}^q\sum_{(\beta_1,\beta_2,\cdots,\beta_q)\in \Lambda_q^k} C_{q,k}\int_0^\infty  e^{it\phi(2^{2j}r)}\\
     &\qquad \qquad\times \prod_{j=1}^q\frac{d^{\beta_j}}{dr^{\beta_j}}\left(\frac{1}{\phi'(2^{2j}r)}\right)\frac{d^{q-k}}{dr^{q-k}}\left((r|s|)^{-\frac{p-2}{2}}J_\frac{p-2}{2}(r|s|)R(r)\right.\\
    &\qquad \qquad\times \left.\mathfrak{L}_m^{(d-
    1)}\left(\frac{r}{2}|z|^2\right)r^{d+p-1}\right)
    \,dr,
    \end{aligned}
	\end{equation*}
 where $\Lambda_q^k=\{(\beta_1,\beta_2,\cdots,\beta_q)\in\{0,1,2,\cdots,k\}^q: \sum\limits_{j=1}^q\beta_j=k\}$. Combined it with \eqref{small-s} and \eqref{phase-prime}, it yields that
 \begin{equation}\label{sharp-s-small}
|\uppercase\expandafter{\romannumeral2}_{m,j}(t,z,s)|\leq C_{q,d,p}|t|^{-q}2^{j(N-2m_1q)}(2m+d)^{-p-1/4}, \,\forall 0\leq q\leq d.
 \end{equation}
\textbf{Case 2.} $|s|>2$. In this case, we shall apply the decay property of Bessel functions.
From \eqref{Bessel-Fourier}, we have
\begin{align*}
&\quad \uppercase\expandafter{\romannumeral2}_{m,j}(t,z,s)\\
&=\frac{C_p2^{Nj}}{(2m+d)^{d+p}}\int_0^\infty  e^{it\phi(2^{2j}r)}\left(e^{ir|s|}h(r|s|)+e^{-ir|s|}\overline{h(r|s|)}\right)\\
              &\qquad \qquad \times R(r) \mathfrak{L}_m^{(d-
    1)}\left(\frac{r}{2}|z|^2\right)r^{d+p-1}
    \,dr\\
    &=\frac{C_p2^{Nj}}{(2m+d)^{d+p}}\int_0^\infty  e^{i\left(t\phi(2^{2j}r)+r|s|\right)}h_+(r|s|) R(r) \mathfrak{L}_m^{(d-
    1)}\left(\frac{r}{2}|z|^2\right)r^{d+p-1}
    \,dr\\
    &\qquad +\frac{C_p2^{Nj}}{(2m+d)^{d+p}}\int_0^\infty  e^{i\left(t\phi(2^{2j}r)-r|s|\right)}h_-(r|s|) R(r) \mathfrak{L}_m^{(d-
    1)}\left(\frac{r}{2}|z|^2\right)r^{d+p-1}
    \,dr\\
    &=\sum_{\pm}\uppercase\expandafter{\romannumeral2}^\pm_{m,j}(t,z,s),
\end{align*}
where
\begin{align*}
&\quad \uppercase\expandafter{\romannumeral2}^\pm_{m,j}(t,z,s)\\
&=\frac{C_p2^{Nj}}{(2m+d)^{d+p}}\int_0^\infty  e^{i\left(t\phi(2^{2j}r)\pm r|s|\right)}h_\pm(r|s|) R(r) \mathfrak{L}_m^{(d-
    1)}\left(\frac{r}{2}|z|^2\right)r^{d+p-1}
    \,dr\\
&=\frac{C_p2^{Nj}}{(2m+d)^{d+p}}\int_0^\infty  e^{it\phi_\pm(2^{2j}r)}h_\pm(r|s|) R(r) \mathfrak{L}_m^{(d-
    1)}\left(\frac{r}{2}|z|^2\right)r^{d+p-1}
    \,dr,
\end{align*}
and $\phi_\pm(r)=\phi(r)\pm \frac{r|s|}{2^{2j}t}$.
Without loss of generality, we can assume that $t>0$ and $\phi'(r)>0$. For $\uppercase\expandafter{\romannumeral2}^+_{m,j}(t,z,s)$, note that $\phi'_+(r)\geq \phi'(r)$ and \eqref{phase-prime} also holds true if we replace $\phi$ by $\phi_+$. Combined with \eqref{big-s}, analogous to Case 1, we obtain
\begin{equation}\label{sharp-s-big-plus}
|\uppercase\expandafter{\romannumeral2}^+_{m,j}(t,z,s)|\leq C_{q,d,p}|t|^{-q}2^{j(N-2m_1q)}(2m+d)^{-p-1/4}, \,\forall 0\leq q\leq d.
 \end{equation}
For $\uppercase\expandafter{\romannumeral2}^-_{m,j}(t,z,s)$, $\phi'_-(2^{2j}r)=\phi'(2^{2j}r)-\frac{|s|}{2^{2j}t}$. Noticing that if $|s|=2^{2j}t\phi'(2^{2j}r)$, then $\phi'_-(2^{2j}r)=0$. We divide the discussion into the following two cases.\\\\
\textbf{Case 2.1.} $|s|\geq 2\sup\limits_{r\in[1/2,2]}2^{2j}t\phi'(2^{2j}r)$. \\
In this case, we have $|\phi'_-(2^{2j}r)|=\frac{1}{2^{2j}t}\left(|s|-2^{2j}t\phi'(2^{2j}r)\right)\geq \phi'(2^{2j}r)$ and \eqref{phase-prime} also holds true if we replace $\phi$ by $\phi_-$. Analogous to $\uppercase\expandafter{\romannumeral2}^+_{m,j}(t,z,s)$, we obtain
\begin{equation}\label{sharp-s-big-minus1}
|\uppercase\expandafter{\romannumeral2}^-_{m,j}(t,z,s)|\leq C_{q,d,p}|t|^{-q}2^{j(N-2m_1q)}(2m+d)^{-p-1/4}, \,\forall 0\leq q\leq d.
 \end{equation}
\\
\textbf{Case 2.2.} $|s|\leq \frac{1}{2}\inf\limits_{r\in[1/2,2]}2^{2j}t\phi'(2^{2j}r)$. \\
 In this case, we have $\phi'_-(2^{2j}r)=\frac{1}{2^{2j}t}\left(2^{2j}t\phi'(2^{2j}r)-|s|\right)\geq \frac{1}{2}\phi'(2^{2j}r)$ and \eqref{phase-prime} also holds true if we replace $\phi$ by $\phi_-$. Analogous to $\uppercase\expandafter{\romannumeral2}^+_{m,j}(t,z,s)$, we obtain
\begin{equation}\label{sharp-s-big-minus2}
|\uppercase\expandafter{\romannumeral2}^-_{m,j}(t,z,s)|\leq C_{q,d,p}|t|^{-q}2^{j(N-2m_1q)}(2m+d)^{-p-1/4}, \,\forall 0\leq q\leq d.
 \end{equation}
 \\
 \textbf{Case 2.3.} $\frac{1}{2}\inf\limits_{r\in[1/2,2]}2^{2j}t\phi'(2^{2j}r)\leq |s|\leq 2\sup\limits_{r\in[1/2,2]}2^{2j}t\phi'(2^{2j}r)$. \\
 In this case, we see that $|s|\sim |t|2^{2jm_1}$. It follows from \eqref{big-s} that
 \begin{equation}\label{sharp-s-big-minus3}
 \begin{aligned}
  |\uppercase\expandafter{\romannumeral2}^-_{m,j}(t,z,s)|&\leq C_{d,p}2^{jN}|s|^{-\frac{p-1}{2}}(2m+d)^{-p-1/4}\\
  &\leq C_{d,p}|t|^{-\frac{p-1}{2}}2^{j\left(N-m_1(p-1)\right)}(2m+d)^{-p-1/4}.
  \end{aligned}
 \end{equation}
 \\
Noting that $p\leq 2d-1$ in Remark \ref{p-d}, from \eqref{sharp-s-small}, \eqref{sharp-s-big-plus}, \eqref{sharp-s-big-minus1}, \eqref{sharp-s-big-minus2} and \eqref{sharp-s-big-minus3}, we have
\begin{equation}\label{sharp}
|\uppercase\expandafter{\romannumeral2}_{m,j}(t,z,s)|\leq C_{d,p}|t|^{-\frac{p-1}{2}}2^{j\left(N-m_1(p-1)\right)}(2m+d)^{-p-1/4}.
 \end{equation}
Interpolating \eqref{trivial} and \eqref{sharp}, we get
\begin{equation*}
|\uppercase\expandafter{\romannumeral2}_{m,j}(t,z,s)|\leq C_{d,p}|t|^{-\theta}2^{j\left(N-2m_1\theta\right)}(2m+d)^{-p-1/4},\,\forall 0\leq\theta\leq\frac{p-1}{2},
 \end{equation*}
which follows immediately that for any $0\leq\theta\leq\frac{p-1}{2}$
\begin{equation*}
|U_t\varphi_j(z,s)|\leq C_{d,p}|t|^{-\theta}2^{j\left(N-2m_1\theta\right)}\sum_{m=0}^\infty(2m+d)^{-p-1/4}\leq C_{d,p}|t|^{-\theta}2^{j\left(N-2m_1\theta\right)}.
\end{equation*}
	
If (H3) holds in addition, we have $\left|\frac{d^2}{dr^2}\left(\phi_-(2^{2j}r)\right)\right|=2^{4j}|\phi''(2^{2j}r)|\sim 2^{2j\alpha_1}$ in Case 2.3. It follows from Lemma \ref{phase} and \eqref{big-s} that
\begin{equation}\label{better-sharp-s-big-minus3}
\begin{aligned}
|\uppercase\expandafter{\romannumeral2}^-_{m,j}(t,z,s)|&\leq C_{d,p}2^{jN}|t2^{2j\alpha_1}|^{-\frac{1}{2}}|s|^{-\frac{p-1}{2}}(2m+d)^{-p-1/4}\\
&\leq C_{d,p} |t|^{-\frac{p}{2}}2^{j\left(N-m_1p-(\alpha_1-m_1)\right)}(2m+d)^{-p-1/4}.
\end{aligned}
 \end{equation}
Using the fact that for any $0\leq \theta\leq 1$, $\frac{p-1+\theta}{2}=(1-\theta)\frac{p-1}{2}+\theta\frac{p}{2}$, by an interpolation between \eqref{sharp-s-big-minus3} and \eqref{better-sharp-s-big-minus3}, we obtain
\begin{equation}\label{Better-sharp-s-big-minus3}
|\uppercase\expandafter{\romannumeral2}^-_{m,j}(t,z,s)|\leq C_{d,p} |t|^{-\frac{p-1+\theta}{2}}2^{j\left(N-m_1(p-1+\theta)-\theta(\alpha_1-m_1)\right)}(2m+d)^{-p-1/4}.
 \end{equation}
Noting that $\alpha_1\leq m_1$, from \eqref{sharp-s-small}, \eqref{sharp-s-big-plus}, \eqref{sharp-s-big-minus1}, \eqref{sharp-s-big-minus2} and \eqref{Better-sharp-s-big-minus3}, we have
\begin{equation}\label{better-sharp}
|\uppercase\expandafter{\romannumeral2}_{m,j}(t,z,s)|\leq C_{d,p} |t|^{-\frac{p-1+\theta}{2}}2^{j\left(N-m_1(p-1+\theta)-\theta(\alpha_1-m_1)\right)}(2m+d)^{-p-1/4},\forall 0\leq\theta\leq 1.
 \end{equation}
It follows immediately that for any $0\leq\theta\leq1$
\begin{equation*}
\begin{aligned}
|U_t\varphi_j(z,s)|&\leq C_{d,p} |t|^{-\frac{p-1+\theta}{2}}2^{j\left(N-m_1(p-1+\theta)-\theta(\alpha_1-m_1)\right)}\sum_{m=0}^\infty(2m+d)^{-p-1/4}\\
&\leq C_{d,p}|t|^{-\frac{p-1+\theta}{2}}2^{j\left(N-m_1(p-1+\theta)-\theta(\alpha_1-m_1)\right)},
\end{aligned}
\end{equation*}
which completes the proof of (a).\\

 The proof of (b) is similar to (a) and we omit the details.\\

 Now we turn to the proof of (c). Take any fixed $\theta$ satisfying $0\leq\theta\leq\min\left(\frac{N}{2m_2},\frac{p-1}{2}\right)$. We write
 \begin{equation}\label{U1}
\begin{aligned}
\|U_t \psi\|_{L^\infty(G)}&=\left\|\sum_{j\leq 2}U_t\Delta_j\psi\right\|_{L^\infty(G)}\\
&=\left\|\sum_{j\leq 2}(U_t\varphi_j)*\psi\right\|_{L^\infty(G)}\\
&\leq\|\psi\|_{L^1(G)}\left\|\sum_{j\leq 2}U_t\varphi_j\right\|_{L^\infty(G)}.
\end{aligned}
\end{equation}
By (b), it gives that
\begin{equation*}
    \left\|\sum_{j\leq 2}U_t\varphi_j\right\|_{L^\infty(G)}\leq \sum_{j\leq 2}\left\|U_t\varphi_j\right\|_{L^\infty(G)}\lesssim |t|^{-\theta}\sum_{j\leq 2}2^{j\left(N-2m_2\theta\right)},
\end{equation*}
which indicates that \eqref{res-sum1} holds when $\theta<\frac{N}{2m_2}$.

Therefore, it suffices to consider the case $\theta=\frac{N}{2m_2}\leq\frac{p-1}{2}$. By \eqref{Imj}, \eqref{IImj} and \eqref{U1}, we have
\begin{equation}\label{U2}
\sum_{j\leq 2}U_t\varphi_j(z,s)=C_{d,p}\sum_{j\leq 2}\sum_{m=0}^\infty \uppercase\expandafter{\romannumeral2}_{m,j}\left(t,\sqrt{\frac{2^{2j}}{2m+d}}z,\frac{2^{2j}}{2m+d}s\right).
\end{equation}
Argued similarly as (a), from the proof of (b), we know that if $j_0\leq 2$ and $2\leq \frac{2^{2j_0}}{2m+d}|s|\sim |t|2^{2j_0m_2}$, then
\begin{equation*}
 \begin{aligned}
  \left|\uppercase\expandafter{\romannumeral2}^-_{m,j_0}\left(t,\sqrt{\frac{2^{2j_0}}{2m+d}}z,\frac{2^{2j_0}}{2m+d}s\right)\right|&\leq C_{d,p}|t|^{-\frac{p-1}{2}}2^{j_0\left(N-m_2(p-1)\right)}(2m+d)^{-p-1/4}\\
  &\leq C_{d,p}|t|^{-\frac{p-1}{2}}\left|\frac{1}{t}\right|^\frac{N-m_2(p-1)}{2m_2}(2m+d)^{-p-1/4}\\
  &\leq C_{d,p}|t|^{-\frac{N}{2m_2}}(2m+d)^{-p-1/4},
  \end{aligned}
 \end{equation*}
 and we deduce
 \begin{equation}\label{sim}
\left|\uppercase\expandafter{\romannumeral2}_{m,j_0}\left(t,\sqrt{\frac{2^{2j_0}}{2m+d}}z,\frac{2^{2j_0}}{2m+d}s\right)\right|\leq C_{d,p}|t|^{-\frac{N}{2m_2}}(2m+d)^{-p-1/4}.
 \end{equation}
When $|j-j_0|>C\gg1$, it also holds true that
 \begin{equation}\label{gg}
\left|\uppercase\expandafter{\romannumeral2}_{m,j}\left(t,\sqrt{\frac{2^{2j}}{2m+d}}z,\frac{2^{2j}}{2m+d}s\right)\right|\leq C_{d,p}|t|^{-q}2^{j(N-2m_2q)}(2m+d)^{-p-1/4}
 \end{equation}
for any $0\leq q\leq d$.\\
Hence, by \eqref{sim} and \eqref{gg} we have
 \begin{align}
  &\quad \sum_{j\leq2}\left|\uppercase\expandafter{\romannumeral2}_{m,j}\left(t,\sqrt{\frac{2^{2j}}{2m+d}}z,\frac{2^{2j}}{2m+d}s\right)\right|\nonumber\\
  &\leq \left(\sum_{|j-j_0|\leq C}+\underset{2^j\leq|t|^{-\frac{1}{2m_2}}}{\sum_{|j-j_0|>C}}+\underset{2^j>|t|^{-\frac{1}{2m_2}}}{\sum_{|j-j_0|>C}}\right)\left|\uppercase\expandafter{\romannumeral2}_{m,j}\left(t,\sqrt{\frac{2^{2j}}{2m+d}}z,\frac{2^{2j}}{2m+d}s\right)\right|\nonumber\\
  &\leq C_{d,p}\left(|t|^{-\frac{N}{2m_2}}+\sum_{2^j\leq|t|^{-\frac{1}{2m_2}}}2^{jN}+\sum_{2^j>|t|^{-\frac{1}{2m_2}}}|t|^{-q}2^{j(N-2m_2q)}\right)(2m+d)^{-p-1/4}\nonumber\\
  &\leq C_{d,p}|t|^{-\frac{N}{2m_2}}(2m+d)^{-p-1/4},\label{last}
  \end{align}
 where the last inequality is obtained by applying Lemma \ref{Sum} and choosing $q=\frac{p+1}{2}\leq d$.  By \eqref{U2} and \eqref{last}, we obtain
 \begin{align*}
\left|\sum_{j\leq 2}U_t\varphi_j(z,s)\right|&\leq C_{d,p}\sum_{m=0}^\infty \sum_{j\leq 2}\left|\uppercase\expandafter{\romannumeral2}_{m,j}\left(t,\sqrt{\frac{2^{2j}}{2m+d}}z,\frac{2^{2j}}{2m+d}s\right)\right|\\
&\leq C_{d,p}|t|^{-\frac{N}{2m_2}}\sum_{m=0}^\infty (2m+d)^{-p-1/4}\\
&\leq C_{d,p}|t|^{-\frac{N}{2m_2}},
\end{align*}
and our desired result immediately comes out from \eqref{U1}.

In addition, if (H4) holds, we can argue similarly and we omit the details.
 	\qed	

\section{Applications}
In this section, we prove Strichartz estimates for some concrete equations on H-type groups by using Theorem \ref{ResultTime}. Above all, we introduce the general Strichartz estimates for $U_t=e^{it\phi(\mathcal{L})}$.

We recall a variant of Hardy-Littlewood-Sobolev inequality.
\begin{lemma}[\cite{GPW2008}]\label{VHLW} Assume $\gamma_1,\gamma_2\in\mathbb{R}$. Let
\begin{equation*}
    k(y)=\begin{cases}
	|y|^{-\theta_1} , \, &|y|\leq1,\\
	|y|^{-\theta_2}, \,&|y|>1.
 \end{cases}
\end{equation*}
Assume that one of the following conditions holds:\\

(1)~$0<\theta_1=\theta_2<d$, $1<r<q<\infty$ and $1-\frac{1}{r}+\frac{1}{q}=\frac{\theta_1}{d}$,\\

(2)~$\theta_1<\theta_2$, $0<\theta_1<d$, $1<r<q<\infty$ and $1-\frac{1}{r}+\frac{1}{q}=\frac{\theta_1}{d}$,\\

(3)~$\theta_1<\theta_2$, $0<\theta_2<d$, $1<r<q<\infty$ and $1-\frac{1}{r}+\frac{1}{q}=\frac{\theta_2}{d}$,\\

(4)~$\theta_1<\theta_2$, $1\leq r\leq q\leq \infty$ and $\frac{\theta_1}{d}<1-\frac{1}{r}+\frac{1}{q}<\frac{\theta_2}{d}$.\\
We have
\begin{equation*}
    \left\|\int_{\mathbb{R}^d}f(\cdot-y)k(y)dy\right\|_{L^q(\mathbb{R}^d)}\lesssim \|f\|_{L^r(\mathbb{R}^d)}.
\end{equation*}
\end{lemma}
\begin{definition}
Given $\theta_1\leq \theta_2\in\mathbb{R}$, we say $\gamma\in E(\theta_1,\theta_2)$ if one of the following holds:\\

(1)~$0<\theta_1=\theta_2<1$ and $\gamma=\frac{2}{\theta_1}$,\\

(2)~$\theta_1<\theta_2$, $0<\theta_1<1$ and $\gamma=\frac{2}{\theta_1}$,\\

(3)~$\theta_1<\theta_2$, $0<\theta_2<1$ and $\gamma=\frac{2}{\theta_2}$,\\

(4)~$\theta_1<\theta_2$, $2\leq \gamma\leq\infty$ and $\theta_1<\frac{2}{\gamma}<\theta_2$.
\end{definition}
We assume that for $2\leq q\leq \infty$, $\rho=\rho(q)\in\mathbb{R}$ and $\theta_1\leq \theta_2$,
\begin{equation}\label{General-Assumption}
    \|U_tf\|_{B_{q,2}^\rho(G)}\lesssim k(t) \|f\|_{B_{q',2}^{-\rho}(G)}, \quad\text{whare } k(t)=\begin{cases}
	|t|^{-\theta_1} , \, &|t|\leq1,\\
	|t|^{-\theta_2}, \,&|t|>1.
	\end{cases}
\end{equation}
Once we obtain this dispersive estimate with different decay rate between $|t|>1$ and $|t|\leq 1$, we prove the following proposition by using Lemma \ref{VHLW} and standard duality arguments in Lemma \ref{equ}-\ref{L1LW}.
\begin{proposition}[see \cite{GPW2008}]\label{General-Strichartz}
    Assume $U_t$ satisfies \eqref{General-Assumption}. For $\gamma\in E(\theta_1,\theta_2)$ and $T>0$, we have
    \begin{align*}
        \|U_tf\|_{L^\gamma\left(\mathbb{R},B_{q,2}^\rho(G)\right)}&\lesssim \|f\|_{L^2(G)},\\
        \|\int_0^tU_{t-\tau}g(\tau,\cdot)d\tau\|_{L^\gamma\left([-T,T],B_{q,2}^\rho(G)\right)}&\lesssim \|g\|_{L^{\gamma'}\left([-T,T],B_{q',2}^{-\rho}(G)\right)},\\
       \|\int_0^tU_{t-\tau}g(\tau,\cdot)d\tau \|_{L^\infty\left([-T,T],H^\rho(G)\right)}&\lesssim \|g\|_{L^{\gamma'}\left([-T,T],B_{q',2}^{-\rho}(G)\right)},\\
    \|\int_0^tU_{t-\tau}g(\tau,\cdot)d\tau\|_{L^\gamma\left([-T,T],B_{q,2}^\rho(G)\right)}&\lesssim \|g\|_{L^1\left([-T,T],L^2(G)\right)}.
    \end{align*}
\end{proposition}
\begin{remark}
    It is worth noticing that Strichartz estimates are widely used for the well-posedness of nonlinear dispersive equations. See the book \cite{Tao2006}. We shall address this problem in the frame of H-type groups in a forthcoming paper.
\end{remark}
\subsection{The fractional Schr\"{o}dinger equation}

\indent First, we consider the fractional Schr\"{o}dinger equation ($0<\alpha<1$)
    \begin{equation}\label{FSchrEqu}
	\begin{cases}
	i\partial_tu+\mathcal{L}^\alpha u=g,\\
	u|_{t=0}=u_0.
	\end{cases}
	\end{equation}
By Duhamel's principle, the solution is formally given by
	\begin{equation*}\label{solution1}
	u(t)=S_tu_0-i\int_0^tS_{t-\tau}g(\tau,\cdot)\,d\tau,
	\end{equation*}
where $S_t=e^{it\mathcal{L}^\alpha}$ and it corresponds to the case when $\phi(r)=r^\alpha$. By a simple calculation,
	\begin{equation*}
	\phi'(r)=\alpha r^{\alpha-1}, \quad
	\phi''(r)=\alpha(\alpha-1)r^{\alpha-2}.
	\end{equation*}
We see that $\phi$ satisfies (H1)-(H4) with $m_1=\alpha_1=m_2=\alpha_2=\alpha$.

Now we show that the fractional Schr\"{o}dinger equation is dispersive.
\begin{theorem}
Let $u$ be the solution of the free fractional Schr\"{o}dinger equation \eqref{FSchrEqu} (i.e. with $g=0$). For $0<\alpha<1$, we have the sharp estimate
	\begin{equation*}
	\|u(t)\|_{L^\infty (G)}\lesssim |t|^{-p/2}\|u_0\|_{\dot{B}^{N-p\alpha}_{1,1}(G)}.
	\end{equation*}
\end{theorem}
\begin{proof}~According to \eqref{res3-3} and \eqref{res-sum2} of Theorem \ref{ResultTime} by taking $\theta=\frac{p}{2}$, we have
	\begin{align*}
	\left\|S_t\Delta_ju_0\right\|_{L^\infty(G)} &\lesssim |t|^{-\frac{p}{2}}2^{j(N-p\alpha)}\|\Delta_ju_0\|_{L^1(G)},\,\forall j\geq 1,\\
 \|S_tP_{\leq 0}u_0\|_{L^\infty (G)}&\lesssim |t|^{-\frac{p}{2}}\|P_{\leq 0}u_0\|_{L^1(G)}
	\end{align*}
Finally, we have the following results:
	\begin{equation*}
	\begin{aligned}
	\|u(t)\|_{L^\infty (G)}&=\|S_tu_0\|_{L^\infty (G)}\\
	&=\|S_tP_{\leq 0}u_0+\sum_{j=1}^\infty S_t\Delta_ju_0\|_{L^\infty (G)} \\
	&\leq \|S_tP_{\leq 0}u_0\|_{L^\infty (G)}+\|\sum_{j=1}^\infty S_t\Delta_ju_0\|_{L^\infty (G)}\\
	&\lesssim |t|^{-\frac{p}{2}}\left(\|P_{\leq 0}u_0\|_{L^1(G)}+\sum_{j=1}^\infty 2^{j(N-p\alpha)}\|\Delta_ju_0\|_{L^1(G)}\right)\\
	&=|t|^{-\frac{p}{2}}\|u_0\|_{B^{N-p\alpha}_{1,1}(G)}.
	\end{aligned}
	\end{equation*}

Finally, let us give an example to explain the sharpness of the decay in time. Let $Q\in C_c^\infty(D_0)$ with $Q(1)=1$, where $D_0$ is a small neighborhood of 1 such that $0\notin D_0$. Then
	\begin{equation*}
	\begin{aligned}
	\widehat{u_0}(m,\lambda)&=Q(|\lambda|)\delta_{m0},   \\
	u_0(z,s)&=C_{d,p}\int_{\mathbb{R}^p}e^{-i\lambda\cdot s-\frac{|\lambda||z|^2}{4}}Q(|\lambda|)|\lambda|^d \,d\lambda.
	\end{aligned}
	\end{equation*}
Consider the free fractional Schr\"{o}dinger equation (\ref{FSchrEqu}). We have
	\begin{equation*}	
u(z,s,t)=e^{it\mathcal{L}^\alpha}u_0=C_{d,p}\int_{\mathbb{R}^p}e^{-i\lambda\cdot s-\frac{|\lambda||z|^2}{4}}e^{it(|\lambda|d)^\alpha }Q(|\lambda|)|\lambda|^d \,d\lambda.
	\end{equation*}
In particular, for a fixed $s_0=\alpha d^\alpha(0,\cdots,0,1)$,
	\begin{equation*}
	u(0,ts_0,t)=C_{d,p}\int_{\mathbb{R}^p}e^{it(-\lambda \cdot s_0+|\lambda|^\alpha d^\alpha)}Q(|\lambda|)|\lambda|^d\,d\lambda.
	\end{equation*}
This oscillating integral has a phase $\Psi(\lambda)=-\lambda\cdot s_0+|\lambda|^\alpha d^\alpha$ with a unique critical point $\lambda_0=\frac{s_0}{\alpha d^\alpha}=(0,\cdots,0,1)$, which is not degenerate. Indeed, the Hessian is equal to
  \begin{equation*}
     H(\lambda_0)= \alpha d^\alpha |\lambda|^{\alpha-2}\left(\delta_{ij}+\frac{\alpha-2}{|\lambda|^2}\lambda_i\lambda_j\right)_{ij}\Big|_{\lambda=\lambda_0}=\alpha d^\alpha
      \left\{
\begin{array}{llll}
1 & & & \\
& \ddots& &\\
& & 1 & \\
& & & \alpha-1 \\
\end{array}
\right\}.
  \end{equation*}
By Lemma \ref{asymptotic}, we get
	\begin{equation*}
	u(0,ts_0,t) \sim |t|^{-p/2}.
	\end{equation*}
\end{proof}

We obtain the decay estimate for $S_t=e^{it\mathcal{L}^\alpha}$.
\begin{proposition}
Let $0<\alpha<1$. Assume $2\leq q\leq\infty$, $1\leq r\leq\infty$, $\delta=\frac{1}{2}-\frac{1}{q}$ and $\rho\leq -\left(N-p\alpha\right)\delta$, then we have
\begin{equation*}
    \|S_tf\|_{B_{q,r}^\rho(G)}\lesssim k(t)\|f\|_{B_{q',r}^{-\rho}(G)}, \quad k(t)=\begin{cases}
	|t|^{-\max (\frac{\rho+N\delta}{\alpha},0)} , \, &|t|\leq1,\\
	|t|^{-p\delta}, \,&|t|\geq1.
	\end{cases}
\end{equation*}
\end{proposition}
\begin{proof}
First we prove the case $|t|\geq 1$. It yields from \eqref{res-sum2} of Theorem \ref{ResultTime} by taking $\theta=\frac{p}{2}$ and Plancherel's identity that
\begin{align*}
 \|S_tP_{\leq 0}u_0\|_{L^\infty (G)}&\lesssim |t|^{-\frac{p}{2}}\|P_{\leq 0}u_0\|_{L^1(G)},\\
 \|S_tP_{\leq 0}u_0\|_{L^2 (G)}&\leq \|P_{\leq 0}u_0\|_{L^2(G)}.
	\end{align*}
 It follows from \eqref{res3-3} of Theorem \ref{ResultTime} by taking $\theta=\frac{p}{2}$ and Plancherel's identity that for $j\geq 1$
 \begin{align*}
	\left\|S_t\Delta_ju_0\right\|_{L^\infty(G)} &\lesssim |t|^{-\frac{p}{2}}2^{j(N-p\alpha)}\|\Delta_ju_0\|_{L^1(G)},\\
\left\|S_t\Delta_ju_0\right\|_{L^2(G)} &\leq \|\Delta_ju_0\|_{L^2(G)}.
	\end{align*}
Applying Riesz-Thorin interpolation theorem, for any $j\geq 1$, we have
\begin{align*}
	\|S_tP_{\leq 0}f\|_{L^q (G)}&\lesssim |t|^{-p\delta}\|P_{\leq 0}f\|_{L^{q'}(G)},\\
\|S_t\Delta_jf\|_{L^q (G)}&\lesssim |t|^{-p\delta}2^{2j\left(N-p\alpha\right)\delta}\|\Delta_jf\|_{L^{q'}(G)}\leq |t|^{-p\delta}2^{-2j\rho}\|\Delta_jf\|_{L^{q'}(G)}.
	\end{align*}
Hence,
\begin{equation*}
    \|S_tf\|_{B_{q,r}^\rho(G)}\lesssim |t|^{-p\delta}\|f\|_{B_{q',r}^{-\rho}(G)},\,|t|\geq1.
\end{equation*}

Next we discuss the case $|t|\leq 1$. From \eqref{res-sum2} of Theorem \ref{ResultTime} by taking $\theta=0$, Plancherel's identity and Riesz-Thorin interpolation theorem , we get
\begin{equation*}
    \|S_tP_{\leq 0}f\|_{L^q(G)}\lesssim \|P_{\leq 0}f\|_{L^{q'}(G)}.
\end{equation*}
For $j\geq1$, it follows from \eqref{res3-3} of Theorem \ref{ResultTime} and Plancherel's identity and Riesz-Thorin interpolation theorem that for $0\leq \theta\leq \frac{p}{2}$
\begin{equation*}
\|S_t\Delta_jf\|_{L^q (G)}\lesssim |t|^{-2\theta\delta}2^{2j(N-2\theta\alpha)\delta}\|\Delta_jf\|_{L^{q'}(G)},
\end{equation*}
which indicates
\begin{equation*}
2^{j\rho}\|S_t\Delta_jf\|_{L^q (G)}\lesssim |t|^{-2\theta\delta}2^{2j\left(\rho+(N-2\theta\alpha)\delta\right)}2^{-j\rho}\|\Delta_jf\|_{L^{q'}(G)}.
\end{equation*}
If $-N\delta \leq \rho\leq -\left(N-p\alpha\right)\delta$, then we can choose $0\leq \theta\leq \frac{p}{2}$ such that $\rho+N\delta=2\theta\alpha \delta$. If $\rho<-N\delta$, we take $\theta=0$. Therefore, we have
\begin{equation*}
2^{j\rho}\|S_t\Delta_jf\|_{L^q (G)}\lesssim |t|^{-\max (\frac{\rho+N\delta}{\alpha},0)}2^{-j\rho}\|\Delta_jf\|_{L^{q'}(G)},
\end{equation*}
and
\begin{equation*}
 \|S_tf\|_{B_{q,r}^\rho(G)}\lesssim |t|^{-\max (\frac{\rho+N\delta}{\alpha},0)}\|f\|_{B_{q',r}^{-\rho}(G)},\,|t|\leq 1,
\end{equation*}
which completes the proof of the proposition.

\end{proof}
\begin{proof}

First we prove the case $|t|\geq 1$. It yields from \eqref{res-sum2} of Theorem \ref{ResultTime} by taking $\theta=\frac{p}{2}$, Plancherel's identity and Riesz-Thorin interpolation theorem, we have
\begin{equation*}
    \|S_tP_{\leq 0}f\|_{L^q (G)}\lesssim |t|^{-p\delta}\|P_{\leq 0}f\|_{L^{q'}(G)}.
\end{equation*}
It follows from \eqref{res3-3} of Theorem \ref{ResultTime} by taking $\theta=\frac{p}{2}$, Plancherel's identity and Riesz-Thorin interpolation theorem that for $j\geq 1$
\begin{align*}
    \|S_t\Delta_jf\|_{L^q (G)}&\lesssim |t|^{-p\delta}2^{2j\left(N-p\alpha\right)\delta}\|\Delta_jf\|_{L^{q'}(G)}\\
    &\lesssim |t|^{-p\delta}2^{-2j\rho}\|\Delta_jf\|_{L^{q'}(G)}.
\end{align*}
Hence,
\begin{equation*}
    \|S_tf\|_{B_{q,r}^\rho(G)}\lesssim |t|^{-p\delta}\|f\|_{B_{q',r}^{-\rho}(G)},\,|t|\geq1.
\end{equation*}

Next we discuss the case $|t|\leq 1$. From \eqref{res-sum2} of Theorem \ref{ResultTime} by taking $\theta=0$, Plancherel's identity and Riesz-Thorin interpolation theorem , we get
\begin{equation*}
    \|S_tP_{\leq 0}f\|_{L^q(G)}\lesssim \|P_{\leq 0}f\|_{L^{q'}(G)}.
\end{equation*}
For $j\geq1$, it follows from \eqref{res3-3} of Theorem \ref{ResultTime} and Plancherel's identity and Riesz-Thorin interpolation theorem that for $0\leq \theta\leq \frac{p}{2}$
\begin{equation*}
\|S_t\Delta_jf\|_{L^q (G)}\lesssim |t|^{-2\theta\delta}2^{2j(N-2\theta\alpha)\delta}\|\Delta_jf\|_{L^{q'}(G)},
\end{equation*}
which indicates
\begin{equation*}
2^{k\rho}\|S_t\Delta_jf\|_{L^q (G)}\lesssim |t|^{-2\theta\delta}2^{j\left(2\rho+2(N-2\theta\alpha)\delta\right)}2^{-k\rho}\|\Delta_jf\|_{L^{q'}(G)}.
\end{equation*}
If $-N\delta \leq \rho\leq -\left(N-p\alpha\right)\delta$, then we can choose $0\leq \theta\leq \frac{p}{2}$ such that $\rho+N\delta=2\theta\alpha \delta$. If $\rho<-N\delta$, we take $\theta=0$. Therefore, we have
\begin{equation*}
2^{k\rho}\|S_t\Delta_jf\|_{L^q (G)}\lesssim |t|^{-\max (\frac{\rho+N\delta}{\alpha},0)}2^{-k\rho}\|\Delta_jf\|_{L^{q'}(G)},
\end{equation*}
and
\begin{equation*}
 \|S_tf\|_{B_{q,r}^\rho(G)}\lesssim |t|^{-\max (\frac{\rho+N\delta}{\alpha},0)}\|f\|_{B_{q',r}^{-\rho}(G)},\,|t|\leq 1,
\end{equation*}
which completes the proof of the proposition.
\end{proof}

By Proposition \ref{General-Strichartz}, we obtain the Strichartz estimates for the fractional  Schr\"{o}dinger operator.
\begin{proposition}\label{Intermediate} Let $0<\alpha<1$. For $i=1,2$, let $\gamma_i, q_i\in[2,\infty]$ and $\rho_i\in\mathbb{R}$ such that
	\begin{equation*}
	\begin{aligned}
	&(i)\,\frac{2}{\gamma_i}+\frac{p}{q_i} \leq \frac{p}{2},   \\
	&(ii)\,\rho_i\leq -(N-\alpha p)\left(\frac{1}{2}-\frac{1}{q_i}\right),
	\end{aligned}
	\end{equation*}
 except for $(\gamma_i, q_i, p)=(2,\infty,2)$. Then for any $T>0$, the fractional Schr\"{o}dinger operator $S_t$ satisfies the estimates
	\begin{equation}\label{Strichartz1}
	\|S_tu_0\|_{L^{\gamma_1}(\mathbb{R}, B_{q_1,2}^{\rho_1}(G))}\lesssim\|u_0\|_{L^2(G)},
	\end{equation}
	\begin{equation}\label{Strichartz2}
	\left\|\int_0^tS_{t-\tau}g(\tau,\cdot)\,d\tau\right\|_{L^{\gamma_1}\left([-T,T], B^{\rho_1}_{q_1,2}(G)\right)}\lesssim \|g\|_{L^{\gamma'_2}([-T,T],B_{q'_2,2}^{-\rho_2}(G))}.
	\end{equation}
\end{proposition}

\begin{theorem}
Under the same hypotheses as in Proposition \ref{Intermediate}, the solution $u$ of
the fractional Schr\"{o}dinger equation \eqref{FSchrEqu} satisfies the following estimate
	\begin{equation*}
	\|u\|_{L^{\gamma_1}([-T,T], B^{\rho_1}_{q_1,2}(G))} \lesssim \|u_0\|_{L^2(G)}+\|g\|_{L^{\gamma'_2}([-T,T],B_{q_2,2}^{-\rho_2}(G))}.
	\end{equation*}
\end{theorem}

Finally, by the embedding relation between the  Besov space and Lebesgue space, we prove the Strichartz inequalities on Lebesgue spaces.
\begin{corollary}\label{SLebesgue} Let $0<\alpha<1$ and $u$ be the solution of the fractional Schr\"{o}dinger equation \eqref{FSchrEqu}. Suppose $\gamma\in[2(N-p\alpha)/p,+\infty)$ and $q\geq2$ satisfy
	\begin{equation*}
	\frac{2\alpha}{\gamma}+\frac{N}{q}=\dfrac{N}{2}-1.
	\end{equation*}
Then we have the following estimate
	\begin{equation*}
	\|u\|_{L^\gamma([-T,T],L^q(G))} \lesssim \|u_0\|_{H^1(G)}+\|g\|_{L^1([-T,T],H^1(G))}.
	\end{equation*}
\end{corollary}
\begin{proof} Under the same hypotheses as in Proposition \ref{Intermediate} and according to \eqref{Strichartz1} and \eqref{Strichartz2}, we have
	\begin{equation*}	
	\begin{aligned}
	&\quad\|u\|_{L^\gamma([-T,T],B^{\rho+1}_{q,2}(G))} \\&\leq \left\|S_tu_0\right\|_{L^\gamma([-T,T],B^{\rho+1}_{q,2}(G))}+\left\|\int_{0}^{t}S_{t-\tau}g(\tau,\cdot)d\tau\right\|_{L^\gamma([-T,T],B^{\rho+1}_{q,2}(G))}.  \\
	 &\lesssim \|u_0\|_{H^1(G)}+\|g\|_{L^1([-T,T],H^1(G))}.
	\end{aligned}
	\end{equation*}
Combining with the fact $B^{\rho+1}_{q,2}(G)\subset L^{q_1}(G)$ where $\rho+1\geq 0$ and $\frac{1}{q_1}=\frac{1}{q}-\frac{\rho+1}{N}$, we have
	\begin{equation*}
	\|u\|_{L^\gamma([-T,T],L^{q_1}(G))} \leq C\left(\|u_0\|_{H^1(G)}+\|g\|_{L^1([-T,T],H^1(G))}\right).
	\end{equation*}
 \end{proof}

\begin{remark}
	When $\alpha=\frac{1}{2}$, we recover the Strichartz estimates for the wave operator on the Heisenberg groups in \cite{BGX2000} and on H-type groups in \cite{H2005}.
\end{remark}
\noindent\\[2mm]
\subsection{The fourth-order Schr\"{o}dinger equation}
Besides, we consider the fourth-order Schr\"{o}dinger equation
	\begin{equation}\label{SSchrEqu}
	\begin{cases}
	i\partial_tu+\mathcal{L}^2u+\mathcal{L}u=g,\\
	u|_{t=0}=u_0.
	\end{cases}
	\end{equation}
By Duhamel's principle, the solution is formally given by
	\begin{equation*}\label{solution3}
	u(t)=U_tu_0-i\int_0^tU_{t-\tau}g(\tau,\cdot)\,d\tau,
	\end{equation*}
where $U_t=e^{it(\mathcal{L}^2+\mathcal{L})}$ and it corresponds to the case when $\phi(r)=r^2+r$. By a simple calculation,
	\begin{equation*}
	\phi'(r)=2r+1, \quad
	\phi''(r)=2.
	\end{equation*}
We know $\phi$ satisfies (H1)–(H4) with $m_1=\alpha_1=\alpha_2=2$, $m_2=1$.

\begin{theorem}
Let $u$ be the solution of the free fourth-order Schr\"{o}dinger equation \eqref{FSchrEqu} (i.e. with $g=0$). We have the following sharp estimate
	\begin{equation*}
	\|u(t)\|_{L^\infty (G)}\lesssim|t|^{-\frac{p}{2}}\|u_0\|_{\dot{B}^{N-2p}_{1,1}(G)}.
	\end{equation*}
\end{theorem}
\begin{proof}
Since the proof of the inequality is similar to the case of the fractional Schr\"{o}dinger equation, we only show the sharpness of the decay. For the same $Q$, $D_0$ and $u_0$, and if $u$ is the solution of \eqref{SSchrEqu} when $g=0$, we have
	\begin{equation*}	
	u(z,s,t)=e^{it(\mathcal{L}^2+\mathcal{L})}u_0=C_{d,p}\int_{\mathbb{R}^p}e^{-i\lambda\cdot s-\frac{|\lambda||z|^2}{4}}e^{it\left(d^2|\lambda|^2+d|\lambda|\right)}Q(|\lambda|)|\lambda|^d \,d\lambda.
	\end{equation*}
In particular,
	\begin{equation*}
	u(0,ts,t)=C_{d,p}\int_{\mathbb{R}^p}e^{it(-\lambda\cdot s+d^2|\lambda|^2+d|\lambda|)}Q(|\lambda|)|\lambda|^d \,d\lambda.\end{equation*}
Consider $u(0,ts_0,t)$ for some fixed $s_0=3d(0,\cdots,0,1)$ such that $|s_0|=3d$.
This oscillating integral has a phase $\Psi(\lambda)=-\lambda\cdot s_0+d^2|\lambda|^2+d|\lambda|$ with a unique critical point $\lambda_0=\frac{s_0}{3d^2}=\frac{1}{d}(0,\cdots,0,1)$, which is not degenerate. Indeed, the Hessian is equal to
  \begin{equation*}
     H(\lambda_0)= \left(\left(2d^2+\frac{d}{|\lambda|}\right)\delta_{ij}-\frac{d}{|\lambda|^3}\lambda_i\lambda_j\right)_{ij}\Big|_{\lambda=\lambda_0}=d^2
      \left\{
\begin{array}{llll}
3 & & & \\
& \ddots& &\\
& & 3 & \\
& & & 2 \\
\end{array}
\right\}.
  \end{equation*}
Applying Lemma \ref{asymptotic}, we have
	\begin{equation*}
	u(0,ts_0,t) \sim |t|^{-\frac{p}{2}}.
	\end{equation*}
\end{proof}
\begin{proposition}
Assume $2\leq q\leq\infty$, $1\leq r\leq\infty$, $\delta=\frac{1}{2}-\frac{1}{q}$ and $\rho\leq -\left(N-2p\right)\delta$. We have
\begin{equation*}
    \|U_tf\|_{B_{q,r}^\rho(G)}\lesssim k(t)\|f\|_{B_{q',r}^{-\rho}(G)}, \quad k(t)=\begin{cases}
	|t|^{-\max (\frac{\rho+N\delta}{2},0)} , \, &|t|\leq1,\\
	|t|^{-p\delta}, \,&|t|\geq1.
	\end{cases}
\end{equation*}
\end{proposition}
\begin{proof}
First we prove the case $|t|\geq 1$. It yields from \eqref{res-sum2-n} of Theorem \ref{ResultTime}, Plancherel's identity and Riesz-Thorin interpolation theorem, we have
\begin{equation*}
    \|U_tP_{\leq 0}f\|_{L^q (G)}\lesssim |t|^{-p\delta}\|P_{\leq 0}f\|_{L^{q'}(G)}.
\end{equation*}
It follows from \eqref{res3-3} of Theorem \ref{ResultTime} by taking $\theta=1$, Plancherel's identity and Riesz-Thorin interpolation theorem that for $j\geq 1$
\begin{align*}
    \|U_t\Delta_jf\|_{L^q (G)}&\lesssim |t|^{-p\delta}2^{2j\left(N-2p\right)\delta}\|\Delta_jf\|_{L^{q'}(G)}\\
    &\lesssim |t|^{-p\delta}2^{-2j\rho}\|\Delta_jf\|_{L^{q'}(G)}.
\end{align*}
Hence,
\begin{equation*}
    \|U_tf\|_{B_{q,r}^\rho(G)}\lesssim |t|^{-p\delta}\|f\|_{B_{q',r}^{-\rho}(G)},\,|t|\geq1.
\end{equation*}

Next we discuss the case $|t|\leq 1$. From \eqref{res-sum1} of Theorem \ref{ResultTime} by taking $\theta=0$, Plancherel's identity and Riesz-Thorin interpolation theorem , we get
\begin{equation*}
    \|U_tP_{\leq 0}f\|_{L^q(G)}\lesssim \|P_{\leq 0}f\|_{L^{q'}(G)}.
\end{equation*}
For $j\geq1$, it follows from \eqref{res3-2} and \eqref{res3-3} of Theorem \ref{ResultTime}, Plancherel's identity and Riesz-Thorin interpolation theorem that for $0\leq \theta\leq \frac{p}{2}$
\begin{equation*}
\|U_t\Delta_jf\|_{L^q (G)}\lesssim |t|^{-2\theta\delta}2^{2j(N-4\theta)\delta}\|\Delta_jf\|_{L^{q'}(G)},
\end{equation*}
which indicates
\begin{equation*}
2^{j\rho}\|U_t\Delta_jf\|_{L^q (G)}\lesssim |t|^{-2\theta\delta}2^{2j\left(\rho+(N-4\theta)\delta\right)}2^{-j\rho}\|\Delta_jf\|_{L^{q'}(G)}.
\end{equation*}
If $-N\delta \leq \rho\leq -\left(N-2p\right)\delta$, then we can choose $0\leq \theta\leq \frac{p}{2}$ such that $\rho+N\delta=4\theta \delta$. If $\rho<-N\delta$, we take $\theta=0$. Therefore, we have
\begin{equation*}
2^{j\rho}\|U_t\Delta_jf\|_{L^q (G)}\lesssim |t|^{-\max (\frac{\rho+N\delta}{2},0)}2^{-j\rho}\|\Delta_jf\|_{L^{q'}(G)},
\end{equation*}
and
\begin{equation*}
 \|U_tf\|_{B_{q,r}^\rho(G)}\lesssim |t|^{-\max (\frac{\rho+N\delta}{2},0)}\|f\|_{B_{q',r}^{-\rho}(G)},\,|t|\leq 1,
\end{equation*}
which completes the proof of the proposition.
\end{proof}

By Proposition \ref{General-Strichartz}, we obtain the Strichartz estimates for the fourth-order Schr\"odinger operator.
\begin{proposition}\label{SResultSB} For $i=1,2$, let $\gamma_i, q_i\in[2,\infty]$ and $\rho_i\in\mathbb{R}$ such that
	\begin{equation*}
	\begin{aligned}
	&(i)\,\frac{2}{\gamma_i}+\frac{p}{q_i} \leq \frac{p}{2},   \\
	&(ii)\,\rho_i\leq -(N-2p)\left(\frac{1}{2}-\frac{1}{q_i}\right),
	\end{aligned}
	\end{equation*}
 except for $(\gamma_i, q_i, p)=(2,\infty,2)$. Then for any $T>0$, the fourth-order Schr\"odinger operator $U_t$ satisfies the estimates
	\begin{equation*}
	\|U_tu_0\|_{L^{\gamma_1}(\mathbb{R}, B_{q_1,2}^{\rho_1}(G))}\lesssim\|u_0\|_{L^2(G)},
	\end{equation*}
	\begin{equation*}
	\left\|\int_0^tU_{t-\tau}g(\tau,\cdot)\,d\tau\right\|_{L^{\gamma_1}\left([-T,T], B^{\rho_1}_{q_1,2}(G)\right)}\lesssim \|g\|_{L^{\gamma'_2}([-T,T],B_{q'_2,2}^{-\rho_2}(G))}.
	\end{equation*}
\end{proposition}

\begin{theorem}\label{Fsolution}
Under the same hypotheses as in Proposition \ref{SResultSB}, the solution $u$ of
the fourth-order Schr\"{o}dinger equation \eqref{SSchrEqu} satisfies the following estimate
	\begin{equation*}
	\|u\|_{L^{\gamma_1}\left([-T,T], B^{\rho_1}_{q_1,2}(G)\right)}\lesssim\|u_0\|_{L^2(G)}+\|g\|_{L^{\gamma'_2}([-T,T],B_{q'_2,2}^{-\rho_2}(G))}.
	\end{equation*}
\end{theorem}

\begin{corollary}\label{FLebesgue}Suppose $\gamma\in[2(N-2p)/p,+\infty]$ and $q\geq2$ satisfy
	\begin{equation*}
	\frac{4}{\gamma}+\frac{N}{q}=\frac{N}{2}-1.
	\end{equation*}
If $u$ is a solution of the fourth-order Schr\"{o}dinger equation \eqref{SSchrEqu}, then the following inequality holds true
	\begin{equation*}
	\|u\|_{L^\gamma([-T,T],L^q(G))} \lesssim\|u_0\|_{H^1(G)}+\|g\|_{L^1([-T,T],H^1(G))}.
	\end{equation*}
\end{corollary}

The proof of Proposition Corollary \ref{FLebesgue} is similar to the case of the fractional Schr\"{o}dinger case. We omit it here.

\noindent\\[2mm]
\subsection{The beam equation}
Moreover, we consider the beam equation
	\begin{equation}\label{BeamEqu}
	\begin{cases}
	\partial_t^2u+\mathcal{L}^2 u+u=g,\\
	u|_{t=0}=u_0 ,  \\
	\partial_tu|_{t=0}=u_1.
	\end{cases}
	\end{equation}
By Duhamel's principle, the solution is formally given by
\begin{equation*}\label{solution3}
	u(t)=\frac{d\mathcal{B}_t}{dt}u_0+\mathcal{B}_tu_1-\int_0^t\mathcal{B}_{t-\tau}g(\tau,\cdot)\,d\tau,
	\end{equation*}
where
	\begin{equation*}
	\mathcal{B}_t=\frac{sin(t\sqrt{I+\mathcal{L}^2})}{\sqrt{I+\mathcal{L}^2}},
	\quad  \frac{d\mathcal{B}_t}{dt}=cos(t\sqrt{I+\mathcal{L}^2}).
	\end{equation*}
So we naturally introduce the operator $B_t=e^{it\sqrt{I+\mathcal{L}^2}}$, which corresponds to the case when $\phi(r)=\sqrt{1+r^2}$.
By a simple calculation,
	\begin{equation*}
	\phi'(r)=r(1+r^2)^{-\frac{1}{2}}, \quad
	\phi''(r)=(1+r^2)^{-\frac{3}{2}}.
	\end{equation*}
We know $\phi$ satisfies (H1)–(H4) with $m_1=1,\alpha_1=-1, m_2=\alpha_2=2$.

The beam equation is also dispersive.
\begin{theorem}
    Let $u$ be the solution of the free beam equation \eqref{BeamEqu} (i.e. g=0). We have the sharp estimates
    \begin{equation*}
        \|u(t)\|_{L^\infty(G)}\lesssim |t|^{-\frac{p}{2}}\left(\|u_0\|_{B^{N-p+2}_{1,1}(G)}+\|u_1\|_{B^{N-p+1}_{1,1}(G)}\right).
    \end{equation*}
\end{theorem}
\begin{proof}
Since the proof of the inequality is very similar to the fractional Schr\"odinger equation, we only show the sharpness of the decay. For the same $Q$, $D_0$ and $u_0$, and if $u$ is the solution of \eqref{BeamEqu} when $g=0$ and $u_1=0$, we have
	\begin{equation*}	
u(z,s,t)=\cos(t\sqrt{I+\mathcal{L}^2})u_0=C_{d,p}\int_{\mathbb{R}^p}e^{-i\lambda\cdot s-\frac{|\lambda||z|^2}{4}}\cos(t\sqrt{1+d^2|\lambda|^2}) Q(|\lambda|)|\lambda|^d \,d\lambda.
	\end{equation*}
Consider $u(0,ts_0,t)$ for some fixed $s_0=\frac{d}{\sqrt{2}}(0,0,\cdots,0,1)$ such that $|s_0|=\frac{d}{\sqrt{2}}$. This oscillatory integral has a phase
	 $\Psi_\pm(\lambda)=-\lambda\cdot s_0\pm\sqrt{1+d^2|\lambda|^2}$ with a unique critical point $\lambda^\pm_0=\pm\frac{\sqrt{2}s_0}{d^2}=\pm\frac{1}{d}(0,0,\cdots,0,1)$, which is not degenerate. Indeed, the Hessian is equal to
  \begin{equation*}
     H(\lambda^\pm_0)= \pm \frac{d^2}{\sqrt{1+d^2|\lambda|^2}}\left(\delta_{ij}-\frac{d^2\lambda_i\lambda_j}{1+d^2|\lambda|^2}\right)_{ij}\Big|_{\lambda=\lambda^\pm_0}=\pm \frac{d^2}{\sqrt{2}}
      \left\{
\begin{array}{llll}
1 & & & \\
& \ddots& &\\
& & 1 & \\
& & & \frac{1}{2} \\
\end{array}
\right\}.
  \end{equation*} Applying Lemma \ref{asymptotic}, we have
	\begin{equation*}
	u(0,ts_0,t) \sim |t|^{-p/2}.
	\end{equation*}
\end{proof}
\begin{proposition}
Assume $2\leq q\leq\infty$, $1\leq r\leq\infty$, $\delta=\frac{1}{2}-\frac{1}{q}$.\\

(1)~Let $0\leq\theta\leq 1$ and $\rho=-\left(n+1+\theta\right)\delta$, then we have
\begin{equation*}
    \|B_tf\|_{B_{q,r}^\rho(G)}\lesssim |t|^{-(p-1+\theta)\delta}\|f\|_{B_{q',r}^{-\rho}(G)}.
\end{equation*}

(2)~Let $0\leq\theta\leq p-1$ and $\rho=-\left(n+1+\theta\right)\delta$, then we have
\begin{equation*}
    \|B_tf\|_{B_{q,r}^\rho(G)}\lesssim |t|^{-(p-1-\theta)\delta}\|f\|_{B_{q',r}^{-\rho}(G)}.
\end{equation*}

(3)~In particular, for $0\leq\theta\leq 1$ and $\rho\leq -\left(n+1+\theta\right)\delta$, then we have
\begin{equation*}
    \|B_tf\|_{B_{q,r}^\rho(G)}\lesssim k(t)\|f\|_{B_{q',r}^{-\rho}(G)}, \quad k(t)=\begin{cases}
	|t|^{-\max (\rho+(n+p)\delta,0)} , \, &|t|\leq1,\\
	|t|^{-(p-1+\theta)\delta}, \,&|t|\geq1.
	\end{cases}
\end{equation*}
\end{proposition}
\begin{proof}
First we prove (1). It yields from \eqref{res-sum2} of Theorem \ref{ResultTime}, Plancherel's identity and Riesz-Thorin interpolation theorem that for $0\leq \theta\leq 1$
\begin{equation*}
    \|B_tP_{\leq 0}f\|_{L^q (G)}\lesssim |t|^{-(p-1+\theta)\delta}\|P_{\leq 0}f\|_{L^{q'}(G)}.
\end{equation*}
It follows from \eqref{res3-3} of Theorem \ref{ResultTime}, Plancherel's identity and Riesz-Thorin interpolation theorem that for $j\geq 1$
\begin{equation*}
   \|B_t\Delta_jf\|_{L^q (G)}\lesssim |t|^{-(p-1+\theta)\delta}2^{2j\left(n+1+\theta\right)\delta}\|\Delta_jf\|_{L^{q'}(G)}.
\end{equation*}
Hence, by $\rho=-\left(n+1+\theta\right)\delta$, we complete the proof of (1).

Next, we prove (2). From  \eqref{res3-2} and \eqref{res-sum1} of Theorem \ref{ResultTime}, for $0\leq \theta\leq p-1$, we get
\begin{align*}
    \|B_tP_{\leq 0}f\|_{L^\infty(G)}&\lesssim |t|^{-\frac{p-1-\theta}{2}}\|P_{\leq 0}f\|_{L^1(G)},\\
    \|B_t\Delta_jf\|_{L^\infty (G)}&\lesssim |t|^{-\frac{p-1-\theta}{2}}2^{j\left(n+1+\theta\right)}\|\Delta_jf\|_{L^1(G)}.
\end{align*}
Argued similarly as the proof of (1), we can easily obtain the results of (2).
\vskip 0.5cm

Finally, we prove (3). We divide the proof into three cases.\\

\textbf{Case 1.} $|t|\geq 1$.\\

In this case, it follows immediately from the results in (1) and the inclusion relation between Besov spaces in Proposition \ref{properties}.\\

\textbf{Case 2.} $|t|\leq 1$ and $\rho+(n+p)\delta>0$.\\

In this case,  it follows immediately from the results in (2) and the inclusion relation between Besov spaces in Proposition \ref{properties}.


\textbf{Case 3.} $|t|\leq 1$ and $\rho+(n+p)\delta\leq 0$.\\
In this case, taking $\theta=p-1$ in (2), for $\rho_1=-(n+p)\delta\geq \rho$, we have
\begin{equation*}
    \|B_tf\|_{B_{q,r}^{\rho_1}(G)}\lesssim \|f\|_{B_{q',r}^{-\rho_1}(G)},
\end{equation*}
which immediately induces our desired result by using the inclusion relation between Besov spaces in Proposition \ref{properties} again.
\end{proof}

By Proposition \ref{General-Strichartz}, we obtain the Strichartz estimates for the beam operator.
\begin{proposition} For $i=1,2$, let $\gamma_i, q_i\in[2,\infty]$ and $\rho_i\in\mathbb{R}$ such that
	\begin{equation*}
	\begin{aligned}
	&(i)\,\max\left((p-2)\left(\frac{1}{2}-\frac{1}{q_i}\right),0\right)<\frac{2}{\gamma_i}< p\left(\frac{1}{2}-\frac{1}{q_i}\right) ,   \\
	&(ii)\,\rho_i\leq -(n+2)\left(\frac{1}{2}-\frac{1}{q_i}\right).
	\end{aligned}
	\end{equation*}
Then for any $T>0$, the beam operator $B_t$ satisfies the estimates
	\begin{equation*}
	\|B_tu_0\|_{L^{\gamma_1}(\mathbb{R}, B_{q_1,2}^{\rho_1}(G))}\lesssim \|u_0\|_{L^2(G)},
	\end{equation*}
	\begin{equation*}
	\left\|\int_0^tB_{t-\tau}g(\tau,\cdot)\,d\tau\right\|_{L^{\gamma_1}\left([-T,T], B^{\rho_1}_{q_1,2}(G)\right)}\lesssim \|g\|_{L^{\gamma'_2}([-T,T],B_{q'_2,2}^{-\rho_2}(G))}.
	\end{equation*}
\end{proposition}

\noindent\\[2mm]
\subsection{The Klein-Gordon equation}
Finally, we consider the Klein-Gordon equation
	\begin{equation}\label{K-GEqu}
	\begin{cases}
	\partial_t^2u+\mathcal{L} u+u=g,\\
	u|_{t=0}=u_0 ,  \\
	\partial_tu|_{t=0}=u_1.
	\end{cases}
	\end{equation}
By Duhamel's principle, the solution is formally given by
\begin{equation*}\label{solution3}
	u(t)=\frac{dA_t}{dt}u_0+A_tu_1-\int_0^tA_{t-\tau}g(\tau,\cdot)\,d\tau,
	\end{equation*}
where
	\begin{equation*}
	A_t=\frac{sin(t\sqrt{I+\mathcal{L}})}{\sqrt{I+\mathcal{L}}},
	\quad  \frac{dA_t}{dt}=cos(t\sqrt{I+\mathcal{L}}).
	\end{equation*}
So we naturally introduce the operator $K_t=e^{it\sqrt{I+\mathcal{L}}}$, which corresponds to the case when $\phi(r)=\sqrt{1+r}$.
By a simple calculation,
	\begin{equation*}
	\phi'(r)=(1+r)^{-\frac{1}{2}}, \quad
	\phi''(r)=-\frac{1}{2}(1+r)^{-\frac{3}{2}}.
	\end{equation*}
We know $\phi$ satisfies (H1)–(H4) with $m_1=\alpha_1=\frac{1}{2}, m_2=1, \alpha_2=2$.
The following theorem indicates that Klein-Gordon equation is dispersive.
\begin{theorem}
    Let $u$ be the solution of the free Klein-Gordon equation \eqref{K-GEqu} (i.e. g=0). We have the sharp estimate
    \begin{equation*}
        \|u(t)\|_{L^\infty(G)}\lesssim |t|^{-\frac{p}{2}}\left(\|u_0\|_{B^{N-\frac{p}{2}}_{1,1}(G)}+\|u_1\|_{B^{N-\frac{p}{2}-1}_{1,1}(G)}\right).
    \end{equation*}
\end{theorem}
\begin{proof}
Since the proof of the inequality is very similar to the fractional Schr\"odinger equation, we only show the sharpness of the decay. For the same $Q$, $D_0$ and $u_0$, and if $u$ is the solution of \eqref{BeamEqu} when $g=0$ and $u_1=0$, we have
	\begin{equation*}	
u(z,s,t)=\cos(t\sqrt{I+\mathcal{L}})u_0=C_{d,p}\int_{\mathbb{R}^p}e^{-i\lambda\cdot s-\frac{|\lambda||z|^2}{4}}\cos(t\sqrt{1+d|\lambda|}) Q(|\lambda|)|\lambda|^d \,d\lambda.
	\end{equation*}
Consider $u(0,ts_0,t)$ for some fixed $s_0=\frac{1}{2\sqrt{2}d}(0,\cdots,0,1)$ such that $|s_0|=\frac{1}{2\sqrt{2}d}$. This oscillatory integral has a phase
	 $\Psi_\pm(\lambda)=-\lambda\cdot s_0\pm\sqrt{1+d|\lambda|}$ with a unique critical point $\lambda^\pm_0=\pm\frac{2\sqrt{2}s_0}{d^2}=\pm\frac{1}{d}(0,0,\cdots,0,1)$, which is not degenerate. Indeed, the Hessian is equal to
  \begin{equation*}
     H(\lambda^\pm_0)= \pm \frac{d}{2|\lambda|\sqrt{1+d|\lambda|}}\left(\delta_{ij}-\frac{2+3d|\lambda|}{2|\lambda|^2(1+d|\lambda|)}\lambda_i\lambda_j\right)_{ij}\Big|_{\lambda=\lambda^\pm_0}=\pm \frac{d^2}{2\sqrt{2}}
      \left\{
\begin{array}{llll}
1 & & & \\
& \ddots& &\\
& & 1 & \\
& & & -\frac{1}{4} \\
\end{array}
\right\}.
  \end{equation*} Applying Lemma \ref{asymptotic}, we have
	\begin{equation*}
	u(0,ts_0,t) \sim |t|^{-p/2}.
	\end{equation*}
\end{proof}

\begin{proposition}
Assume $2\leq q\leq\infty$, $1\leq r\leq\infty$, $\delta=\frac{1}{2}-\frac{1}{q}$ and $\rho\leq -\left(N-\frac{p}{2}\right)\delta$, then we have
\begin{equation*}
    \|K_tf\|_{B_{q,r}^\rho(G)}\lesssim k(t)\|f\|_{B_{q',r}^{-\rho}(G)}, \quad k(t)=\begin{cases}
	|t|^{-\max (2(\rho+N\delta),0)} , \, &|t|\leq1,\\
	|t|^{-p\delta}, \,&|t|\geq1.
	\end{cases}
\end{equation*}
\end{proposition}
\begin{proof}
First we prove the case $|t|\geq 1$. It yields from \eqref{res-sum2-n} of Theorem \ref{ResultTime}, Plancherel's identity and Riesz-Thorin interpolation theorem that
\begin{align*}
\|K_tP_{\leq 0}f\|_{L^q (G)}&\lesssim |t|^{-p\delta}\|P_{\leq 0}f\|_{L^{q'}(G)}.
\end{align*}
It follows from \eqref{res3-3} of Theorem \ref{ResultTime} by taking $\theta=1$, Plancherel's identity and Riesz-Thorin interpolation theorem that for $j\geq 1$
\begin{align*}
 \|K_t\Delta_jf\|_{L^q (G)}&\lesssim |t|^{-p\delta}2^{j\left(2N-p\right)\delta}\|\Delta_jf\|_{L^{q'}(G)}\\
    &\lesssim |t|^{-p\delta}2^{-2j\rho}\|\Delta_jf\|_{L^{q'}(G)}.
\end{align*}
Hence,
\begin{equation*}
    \|K_tf\|_{B_{q,r}^\rho(G)}\lesssim |t|^{-p\delta}\|f\|_{B_{q',r}^{-\rho}(G)},\,|t|\geq1.
\end{equation*}

Next, we prove the case $|t|\leq 1$. From \eqref{res-sum1} of Theorem \ref{ResultTime} by taking $\theta=0$, Plancherel's identity and Riesz-Thorin interpolation theorem , we get
\begin{equation*}
    \|K_tP_{\leq 0}f\|_{L^q(G)}\lesssim \|P_{\leq 0}f\|_{L^{q'}(G)}.
\end{equation*}
For $j\geq1$, it follows from \eqref{res3-2} and \eqref{res3-3} in Theorem \ref{ResultTime}, Plancherel's identity and Riesz-Thorin interpolation theorem that for $0\leq \theta\leq \frac{p}{2}$
\begin{equation*}
\|K_t\Delta_jf\|_{L^q (G)}\lesssim |t|^{-2\theta\delta}2^{2j(N-\theta)\delta}\|\Delta_jf\|_{L^{q'}(G)},
\end{equation*}
which indicates
\begin{equation*}
2^{j\rho}\|K_t\Delta_jf\|_{L^q (G)}\lesssim |t|^{-2\theta\delta}2^{2j\left(\rho+(N-\theta)\delta\right)}2^{-j\rho}\|\Delta_jf\|_{L^{q'}(G)}.
\end{equation*}
If $-N\delta \leq \rho\leq -\left(N-\frac{p}{2}\right)\delta$, then we can choose $0\leq \theta\leq \frac{p}{2}$ such that $\rho+N\delta=\theta \delta$. If $\rho<-N\delta$, we take $\theta=0$. Therefore, we have
\begin{equation*}
2^{j\rho}\|K_t\Delta_jf\|_{L^q (G)}\lesssim |t|^{-\max (2(\rho+N\delta),0)}2^{-j\rho}\|\Delta_jf\|_{L^{q'}(G)},
\end{equation*}
and
\begin{equation*}
 \|K_tf\|_{B_{q,r}^\rho(G)}\lesssim |t|^{-\max (2(\rho+N\delta),0)}\|f\|_{B_{q',r}^{-\rho}(G)},\,|t|\leq 1,
\end{equation*}
which completes the proof of the proposition.
\end{proof}

By Proposition \ref{General-Strichartz}, we obtain the Strichartz estimates for the Klein-Gordon operator.
\begin{proposition} For $i=1,2$, let $\gamma_i, q_i\in[2,\infty]$ and $\rho_i\in\mathbb{R}$ such that
	\begin{equation*}
	\begin{aligned}
	&(i)\,\frac{2}{\gamma_i}+\frac{p}{q_i} \leq \frac{p}{2},   \\
	&(ii)\,\rho_i\leq -(N-\frac{p} {2})\left(\frac{1}{2}-\frac{1}{q_i}\right),
	\end{aligned}
	\end{equation*}
 except for $(\gamma_i, q_i, p)=(2,\infty,2)$. Then for any $T>0$, the Klein-Gordon operator $K_t$ satisfies the estimates
	\begin{equation*}
	\|K_tu_0\|_{L^{\gamma_1}(\mathbb{R}, B_{q_1,2}^{\rho_1}(G))}\lesssim\|u_0\|_{L^2(G)},
	\end{equation*}
	\begin{equation*}
	\left\|\int_0^tK_{t-\tau}g(\tau,\cdot)\,d\tau\right\|_{L^{\gamma_1}\left([-T,T], B^{\rho_1}_{q_1,2}(G)\right)}\lesssim \|g\|_{L^{\gamma'_2}([-T,T],B_{q'_2,2}^{-\rho_2}(G))}.
	\end{equation*}
\end{proposition}

\begin{remark}
    As the fractional Schr\"odinger equation and fourth-order Schr\"odinger equation, we can also obtain the Strichartz inequalities for the solution of the beam equation \eqref {BeamEqu} and the Klein-Gordon equation \eqref{K-GEqu}. We omit it here.
\end{remark}

\section*{Acknowledgments} M.S. is supported by the National Natural Science Foundation of China (Grant No. 11701452) and Guangdong Basic and Applied Basic Research Foundation (No. 2023A1515010656).  J.T. is supported by Jiangxi Normal University 12022821.

\end{document}